\documentclass[preprint,11pt]{elsarticle}

\usepackage{amssymb}
\usepackage{enumerate}
\usepackage{amsmath}
\usepackage{amsfonts}
\usepackage{mathrsfs}

\usepackage[dvipsnames,usenames]{xcolor}
\usepackage{hyperref}

\usepackage{lmodern}
\usepackage{diagbox}
\usepackage{multirow}
\usepackage{rotating} 
\newtheorem{theorem}{Theorem}[section]
\newtheorem{lemma}[theorem]{Lemma}
\newtheorem{e-proposition}[theorem]{Proposition}
\newtheorem{corollary}[theorem]{Corollary}
\newtheorem{e-definition}[theorem]{Definition}
\newtheorem{remark}{\it Remark\/}

\numberwithin{equation}{section}
\newenvironment{proof}{\par\noindent\textbf{Proof. }}{\hspace{\stretch{1}}$\square$\medskip\par}

\begin{document}

\begin{frontmatter}

 \author{Ch\'erif Amrouche $^a$}
 \ead{cherif.amrouche@univ-pau.fr}
 \address[authorlabel1]{Laboratoire de Math\'ematiques et Leurs Applications, UMR CNRS 5142\\ Universit\'e de Pau et des Pays de l'Adour,  64000 Pau,  France}
 \author{Mohand Moussaoui $^{b, \star}$}

 \address[authorlabel2]{Laboratoire des \'Equations aux D\'eriv\'ees Partielles Non Lin\'eaires et
Histoire des Math\'ematiques, Ecole Normale Sup\'erieure, Kouba, Algeria}

\title{On the traces of harmonic functions $H^{1/2}$ and  $H^{3/2}$ in Lipschitz domains}

\begin{abstract}
In this work, we revisit the following estimate due to Dahlberg \cite{Dahl}.  Let $\textit{\textbf x}_0$ a fixed point in a bounded Lipschitz domain $\Omega$. Then there exists a constant $C > 0$ such that if $u$ is a harmonic function in $\Omega$ and vanishes at $\textit{\textbf x}_0$, then 
\begin{equation*}
C^{-1} \Vert u \Vert_{L^2(\Gamma)} \leq \Big(\int_\Omega \varrho\vert \nabla u \vert^2\Big)^{1/2} \leq C \Vert u \Vert_{L^2(\Gamma)},
\end{equation*}
where $\varrho$ is the distance to the boundary of $\Omega$. Using Grisvard's work and interpolation theory for subspaces, we complete the solvability of the inhomogeneous Dirichlet problem:
$$
(\mathscr{L}_D^0)\ \ \ \  -\Delta u = f\quad \ \mbox{in}\ \Omega \quad
\mbox{and } \quad u = 0 \ \ \mbox{on }\Gamma,
$$ 
in a framework of fractional Sobolev spaces $H^s(\Omega)$, when $\Omega$ is a polygon or a polyhedron domain and $1/2 \leq s \leq 2$. Thanks to these regularity results  and an explicit function given by Ne$\mathrm{\check{c}}$as, we show that the above inequalities cannot be valid in their current form.  On the other hand, we  identify a functional space which satisfies the embeddings $H^{1/2}_{00}(\Omega)\hookrightarrow E(\nabla;\, \Omega) \hookrightarrow H^{1/2}(\Omega)$
and the trace operator $\gamma_0$  from $E(\nabla;\, \Omega)$ into $L^2(\Gamma)$ is well-defined and continuous. This  leads to an alternative to the functions $H^{1/2}(\Omega)$, non necessarily harmonic, for having a trace in $L^2(\Gamma)$ and also to a new
characterization of $H^{1/2}_{00}(\Omega)$ as the kernel of this operator. However, we show that if the domain $\Omega$  is of class $\mathscr{C}^{1, 1}$, then the above inequalities are valid.
\end{abstract}


\end{frontmatter}

\section{Introduction and motivation}
In \cite{Necas}, Ne\v{c}as proved the following property (see Theorem 2.2 Section 6): if $\varrho^{\alpha/p} u\in L^p(\Omega)$ and $\varrho^{\alpha/p}\nabla u \in L^p(\Omega)$, for some $0 \leq \alpha < p - 1$, then $u_{\vert\Gamma} \in L^p(\Gamma)$ and 
\begin{equation*}\label{inegNec}
\int_\Gamma \vert u \vert^p \leq C(\Omega) \big(\int_\Omega \varrho^\alpha\vert  u \vert^p + \int_\Omega \varrho^\alpha\vert \nabla u \vert^p\big).
\end{equation*}
However, if $\alpha = p - 1$, he showed, using a counterexample with $\Omega= \,  ]0, 1/2[\,  \times\,  ]0, 1/2[ $, that the above inequality does not hold in general. In particular for $\alpha = 1$ and $p = 2$, if $ \sqrt \varrho \, u$ and $\sqrt \varrho \, \nabla u \in L^2(\Omega)$ (which implies that $u\in H^{1/2}(\Omega)$),  the function $u$ may have no trace in $L^2(\Gamma)$. \medskip

What about if in addition the function $u$ is harmonic? In \cite{Dahl}, see Corollary, Section 6,  the author proved the following property: let $\textit{\textbf x}_0$ a fixed point in $\Omega$, then there exists a constant $C > 0$ such that if $u$ is a harmonic function in $\Omega$ and vanishes at $\textit{\textbf x}_0$, then 

\begin{equation}\label{inegaltraceL2Gamma1}
C^{-1} \Vert u \Vert_{L^2(\Gamma)} \leq \Vert \sqrt  \varrho\,  \nabla u  \Vert_{L^{2}(\Omega)} \leq C \Vert u \Vert_{L^2(\Gamma)}.
\end{equation}

In \cite{Dahlberg}, the authors showed in the same context, but with different proof, a result close to the one mentioned above:  there exists $C > 0$ such that
\begin{equation}\label{ineg1}
\int_\Gamma \vert u \vert^2 d\sigma \leq C \int_\Gamma \vert S(u)\vert^2 d\sigma = C \int_\Omega \delta \vert\nabla u\vert^2 dx,
\end{equation}
without any specification of spaces of such function. Here, $\Omega$ is a bounded Lipschitz domain of $\mathbb{R}^N$, $N \geq 2$,  and $S(u)$ is the area integral of $u$ and $\delta$ is an adaptative distance to the boundary, equivalent  to the distance $\varrho$ to the boundary $\Gamma$. \medskip

The proof of \eqref{ineg1} is given on pages 1428 and 1429 in  \cite{Dahlberg} and contains unjustified formal calculations. \smallskip

Recall that if $u\in L^2(\Omega)$, then we have the following implication: 
$$
\sqrt \varrho\,  \nabla u \in L^2(\Omega) \Longrightarrow u \in H^{1/2}(\Omega) \quad \mathrm{and}\quad \sqrt \varrho\,  \nabla^2 u \in L^2(\Omega) \Longrightarrow u \in H^{3/2}(\Omega)
$$
and the reverse implications hold if moreover $u$ is a harmonic function, see Theorem 3.2 and Theorem 3.8 in \cite{AM} or Theorem 4.2 in \cite {J-K}. Here $\nabla^2 u$ denotes the Hessian matrix of $u$. In addition,  we have  the following equivalence norms for harmonic functions:
\begin{equation*}
\begin{array}{rl}
\Vert u \Vert_{H^{1/2}(\Omega)} \approx & \Vert u \Vert_{L^2(\Omega)} + \Vert \sqrt  \varrho\,  \nabla u  \Vert_{L^{2}(\Omega)} \quad \mathrm{for}\; u \in H^{1/2}(\Omega)\cap \mathscr{H}, \\
\Vert u \Vert_{H^{3/2}(\Omega)} \approx & \Vert u \Vert_{L^2(\Omega)} + \Vert \sqrt  \varrho\,  \nabla^2 u  \Vert_{L^{2}(\Omega)} \quad \mathrm{for}\; u \in H^{3/2}(\Omega)\cap \mathscr{H}, 
\end{array}
\end{equation*}
where $ \mathscr{H} $ is the space of harmonic functions in $\Omega$.
Inequalities \eqref{inegaltraceL2Gamma1} would then result in the equivalence of the following norms:
\begin{equation*}
\Vert u \Vert_{H^{1/2}(\Omega)} \approx  \Vert u \Vert_{L^2(\Gamma)} \quad \mathrm{for}\; u \in H^{1/2}(\Omega)\cap \mathscr{H} 
\end{equation*}
and would imply that any harmonic function $H^{1/2}(\Omega)$ has a trace  in $L^2(\Gamma)$.  Consequently, the gradient of any harmonic function $u$ in $H^{3/2}(\Omega)$ would have a trace in $\textit{\textbf L}^2(\Gamma)$, {\it i.e} $u\in H^1(\Gamma)$ and $\partial_\textit{\textbf n} u \in L^2(\Gamma)$.\medskip

In the same spirit, inequalities \eqref{inegaltraceL2Gamma1} would imply the following property: let $u$ be a harmonic function in $\Omega$ satisfying $u(\textit{\textbf x}_0) = 0$ and $\nabla u(\textit{\textbf x}_0) = {\bf 0}$ at some point $\textit{\textbf x}_0\in \Omega$, then
\begin{equation*}
\Vert u \Vert_{H^1(\Gamma)} \leq C(\Omega)\Big(\int_\Omega \varrho\vert \nabla^2 u \vert^2\Big)^{1/2}.
\end{equation*}
And as above,  we would have the equivalence of the following norms:
\begin{equation*}
\Vert u \Vert_{H^{3/2}(\Omega)} \approx  \Vert u \Vert_{H^1(\Gamma)} \quad \mathrm{for}\; u \in H^{3/2}(\Omega)\cap \mathscr{H}.
\end{equation*}

The purpose of this work is to show that the inequalities \eqref{inegaltraceL2Gamma1} and \eqref{ineg1} cannot be valid in their current form and to propose an alternative for the functions $H^{1/2}(\Omega)$, resp. $H^{3/2}(\Omega)$, to have a trace in $L^2(\Gamma)$, resp. in $H^1(\Gamma)$.\medskip
 
The paper is organized as follows. In Section \ref{rap}, we provide some reminders about fractional Sobolev spaces and their properties. In Section \ref{poly}, we revisit and complete the solvability of the inhomogeneous Dirichlet problem:
$$
(\mathscr{L}_D^0)\ \ \ \  -\Delta u = f\quad \ \mbox{in}\ \Omega \quad
\mbox{and } \quad u = 0 \ \ \mbox{on }\Gamma,
$$ 
in a framework of fractional Sobolev spaces when $\Omega$ is a polygon or a polyhedron.  Recall that when  
 $\Omega$ is a general Lipschitz domain, for any $1/2 < s < 3/2$, the operator 
 $$
 \Delta : H^{2- s}_0(\Omega) \longrightarrow H^{-s}(\Omega)
 $$ 
 is an isomorphism. What happens when $0 \leq s \leq 1/2$? Using Grisvard's work and interpolation theory for subspaces, we provide answers to this question  in the case where $\Omega$ is a polygonal domain of $\mathbb{R}^2$  or a polyhedral domain of $\mathbb{R}^3$,  particularly when the domain is non convex. Thanks to the regularity results given in Section \ref{poly} and an explicit function given by Ne$\mathrm{\check{c}}$as, we shall see in Section \ref{counterexample} a counterexample concerning the above inequalities.  It is important to note that the estimates \eqref{inegaltraceL2Gamma1} and \eqref{ineg1} have been widely used, in particular to argue that $H^{3/2}$-regularity for the $H^{1}_0$ solution of the Laplace equation is unattainable for $f$ in the dual of $H^{1/2}_{00}(\Omega)$.  In a forthcoming paper, we show that this regularity is indeed achieved for such RHS, when the bounded domain is only Lipschitz. Section \ref{traces} is devoted to traces in the limit cases $H^{1/2}(\Omega)$ and $H^{3/2}(\Omega)$. We  identify a functional space 
\begin{equation*}
 E(\nabla;\, \Omega)\ =\ \left\{\,v\in H^{1/2}(\Omega);\ \nabla v\in  [\textit{\textbf H}^{\, 1/2}(\Omega)]'\,\right\},
\end{equation*}
which satisfies the embeddings $H^{1/2}_{00}(\Omega)\hookrightarrow E(\nabla;\, \Omega) \hookrightarrow H^{1/2}(\Omega)$ and the trace operator $\gamma_0$  from $E(\nabla;\, \Omega)$ into $L^2(\Gamma)$ is well-defined and continuous, leading to a new
characterization of $H^{1/2}_{00}(\Omega)$ as the kernel of this operator.

\section{Functional framework}\label{rap}

\subsection{Spaces $H^s(\Omega)$}

Let us recall the definitions of some Sobolev spaces and some important properties that will be useful later. Recall first the following Sobolev space: for $s\in \mathbb{R}$,
$$
H^s(\mathbb{R}^N)\ =\ \left\{\, v\in \mathscr{S}'(\mathbb{R}^N);\ (1+|\xi|^2)^{s/2}\widehat{v}\in L^2(\mathbb{R}^N) \right\}
$$
which is a Hilbert space for the norm: $
\Vert u\Vert_{H^s(\mathbb{R}^N)} = \left(\int_{\mathbb{R}^N} (1+|\xi|^2)^{s}\vert\widehat{v}\vert^2 dx\right)^{1/2}. $ Here the notation $\widehat{v}$ denotes the Fourier transform of $v$. For each non negative real $s$, define
$$
H^s(\Omega)\ =\ \left\{\, v|_{\Omega};\ \ v\in H^s(\mathbb{R}^N) \right\},
$$
with the usual quotient norm $
\Vert u\Vert_{H^s(\Omega)} = \inf\{\Vert v\Vert_{H^s(\mathbb{R}^N)} ;\, v|_{\Omega} = u\; \mathrm{in}\; \Omega\}.$ If $m\in\mathbb{N}$, then
$$
H^m(\Omega)\ =\ \left\{\, v\in L^2(\Omega);\ D^\lambda v \in L^2(\Omega)\ \ \textrm{ for } 0<|\lambda|\leq m\right\},
$$
and we have the equivalence:
$$
\Vert u\Vert_{H^m(\Omega)} \approx \left(\sum_{0\leq \vert\lambda\vert \leq m}\Vert  D^\lambda u\Vert^2_{L^2(\Omega)}\right)^{1/2}.
$$
Using the interpolation by the complex method, recall that
$$
H^s(\Omega)\ =\ \left[H^m(\Omega),L^2(\Omega) \right]_\theta \quad \mathrm{with}\quad  s=(1-\theta)m\ \textrm{ and }0<\theta<1.
$$
According to \cite{Ada} and \cite{Gri}, when $s = m + \sigma$, with $0 < \sigma < 1$, $H^s(\Omega)$ can be equipped with an equivalent and intrinsic norm $
\Vert u\Vert_{H^s(\Omega)} = (\Vert u\Vert_{H^m(\Omega)}^2 + \vert u\vert^2_{H^s(\Omega)})^{1/2},$ where
$$
\vert u\vert_{H^s(\Omega)} =  \big(\sum_{\vert\lambda\vert = m}\int_\Omega\int_\Omega\frac{\vert D^\lambda u(x) - D^\lambda u(y)\vert^2}{\vert x - y \vert^{N + 2\sigma}}dx\,dy\big)^{1/2}.
$$

\subsection{Spaces $H^s_0(\Omega)$ and  $H^s_{00}(\Omega)$}

This leads us to introduce the following space
$$
H^s_0(\Omega)\ =\ \overline{\mathscr{D}(\Omega)}^{\, ||\,.\,||_{H^s(\Omega)}}\quad  \textrm{ with } s\geq 0, 
$$
\textit{i.e.}, the closure of the space $\mathscr{D}(\Omega)$ for the norm $||\,.\, ||_{H^s(\Omega)}$. Let us also recall that for any $ 0\leq s\leq 1/2$, the space $\mathscr{D}(\Omega) \textrm{ is dense in } H^s(\Omega)$. That means that $H^s(\Omega) = H^s_0(\Omega)$   for $ 0\leq s\leq 1/2.$  Moreover, we have the following properties:\smallskip

\textit{\textbf {i}}) Let $u\in H^s_0(\Omega)$ with $0 \leq s \le 1$. Then 
\begin{equation*}\label{a2-e5}
\frac{u}{\varrho^s}\in L^2(\Omega)\quad  \textrm{ when } \quad s\neq 1/2.
\end{equation*}
Moreover, we have the following Hardy inequality: 
\begin{equation*}\label{HardyIneq}
\Vert \frac{u}{\varrho^s} \Vert_{L^2(\Omega)} \leq C(\Omega) \vert u \vert_{H^s(\Omega)}.
\end{equation*}
\textit{\textbf {ii}}) More generally, let $u\in H^s_0(\Omega)$ with $s>0$ and such that $ s - 1/2$ is not an integer. Then
\begin{equation}\label{a2-e6}
\forall  |\lambda|\leq s, \quad \frac{D^\lambda u}{\varrho^{s-|\lambda|}}\in L^2(\Omega),
\end{equation}
with similar inequalities as above.\medskip

Let us to introduce the following space: for $ s \geq 0$
$$
 \widetilde{H}^s(\Omega)\ =\ \left\{\, v\in H^s(\Omega);\, \widetilde{v}\in H^s(\mathbb{R}^N)  \right\},
 $$
 where $ \widetilde{v}$ is the extension of $v$ by zero outside $\Omega$. The space $ \widetilde{H}^s(\Omega)$ is a Hilbert for the norm 
 $$
 \Vert u \Vert_{\widetilde{H}^s(\Omega)} =  \Vert \widetilde{u} \Vert_{H^s(\mathbb{R}^N)}
 $$
 and satisfies the following property:
 \begin{equation}\label{a2-e7}
\widetilde{H}^s(\Omega)=H^s_0(\Omega)\quad \mathrm{when}\quad s\notin\left\{1/2\right\}+\mathbb{N}.
\end{equation}

Another way to characterize the space $H^s_0(\Omega)$, for $s>1/2$ and $s\notin\left\{1/2\right\}+\mathbb{N}$, is given by
\begin{equation*}\label{a2-e8}
u\in H^s_0(\Omega)\ \Longleftrightarrow\ u\in H^s(\Omega)\ \textrm{ and }\ \frac{\partial^j u}{\partial \textit{\textbf {n}}\, ^j}=0,\ \  0\leq j\leq s-1/2,
\end{equation*}
where $\textit{\textbf {n}}$ is the outward normal vector to the boundary of $\Omega$.
For the case  $s = 3/2$, we have $H^{3/2}_0(\Omega) = H^{3/2}(\Omega)\cap H^{1}_0(\Omega)$. The interpolation between two spaces $H^s_0(\Omega)$ is somewhat different from the one between two spaces $H^s(\Omega)$. Indeed, if $s_1>s_2\geq 0$ such that $ s_1, s_2\notin\left\{1/2\right\}+\mathbb{N} $, then we have
\begin{equation*}\label{a3-e1}
[H^{s_1}_0(\Omega), H^{s_2}_0(\Omega)]_\theta=H^{(1-\theta)s_1+\theta s_2}_0(\Omega)\ \ \ \  \textrm{ if } \ (1-\theta)s_1+\theta s_2 \notin\left\{1/2\right\}+\mathbb{N} 
\end{equation*}
and 
\begin{equation*}\label{a3-e2}
[H^{s_1}_0(\Omega), H^{s_2}_0(\Omega)]_\theta=H^{(1-\theta)s_1+\theta s_2}_{00}(\Omega)\ \ \ \  \textrm{ if } \ (1-\theta)s_1+\theta s_2 \in\left\{1/2\right\}+\mathbb{N} 
\end{equation*}
where the space $H^s_{00}(\Omega)$ is defined as follows: For any $\mu\in\mathbb{N}$, 
\begin{equation*}\label{a3-e3}
H^{\mu+1/2}_{00}(\Omega)\ =\ \left\{\, u\in H^{\mu+ 1/2}_{0}(\Omega);\ \ \frac{D^\lambda u}{\varrho^{1/2}}\in L^2(\Omega), \ \ \forall |\lambda|=\mu \,\right\}.
\end{equation*}
This is a strict subspace of $H^{\mu+1/2}_{0}(\Omega)$ with a strictly finer topology and $\mathscr{D}(\Omega) $ is dense in  $H^{\mu+1/2}_{00}(\Omega)$  for this finer topology. \medskip
 
The property (\ref{a2-e6}) admits a reciprocal one if $ s \notin\left\{1/2\right\}+\mathbb{N}$:
 \begin{equation*}\label{a3-e4}
 u\in H^s_0(\Omega)\ \Longleftrightarrow\ u\in  L^2(\Omega)\ \textrm{ and } \frac{D^\lambda u}{\varrho^{s-|\lambda|}}\in L^2(\Omega),\ \  \forall |\lambda|\leq s. 
 \end{equation*}
Regarding the property (\ref{a2-e7}), we have 
\begin{equation*}\label{a3-e5}
\widetilde{H}^s(\Omega)=H^s_{00}(\Omega)\quad \mathrm{when}\quad s\in\left\{1/2\right\}+\mathbb{N}.
\end{equation*}
We now have a look at their dual spaces. For $s\geq 0$ and $s \notin\left\{1/2\right\}+\mathbb{N}$, we denote the dual space of $H^s_0(\Omega)$ by $H^{-s}(\Omega)$. Note that since  $\mathscr{D}(\Omega)$ is dense in $\widetilde{H}^s(\Omega)$, then the dual space of $\widetilde{H}^s(\Omega)$ could be identified to a subspace of  $\mathscr{D}'(\Omega)$.

\section{$H^s$-regularity for Laplace equation}\label{poly}

\subsection{Characterization of harmonic functions kernels}\label{ssCharkernels} 

The characterization of the kernels of harmonic functions belonging to Sobolev spaces $H^s(\Omega)$ and satisfying a homogeneous Dirichlet condition is simple in the case where the open set $\Omega$ is sufficiently regular, even in the case where $s$ is a negative real number.  Let us define the following kernel for any $ - 1/2 < s <  \infty$: 
$$
\mathscr{N}_s(\Omega) = \{\varphi \in H^{-s} (\Omega); \; \Delta \varphi = 0 \; \mathrm{in}\; \Omega \; \mathrm{and}\; \varphi = 0\; \mathrm{on}\;  \Gamma \}.
$$
In the case where $\Omega$ is a bounded domain, convex or of class $\mathscr{C}^{1, 1}$ in $\mathbb{R}^N$ with $N \geq 2$,  this kernel is trivial for all $s < 1/2$. The situation is different if the bounded domain $\Omega$ is only Lipschitz. \medskip

As an example, let us consider the following Lipschitz and non convex domain: 
\begin{equation*}
\mathrm{for}  \; 1/2 < \alpha < 1, \quad \Omega = \{(r, \theta);\; 0 < r < 1,\quad 0 < \theta < \frac {\pi} {\alpha}\}.
\end{equation*}
We can easily verify that the following function
\begin{equation*}
z(r, \theta) = (r^{-\alpha} - r^{\alpha})\mathrm{sin}(\alpha\theta)
\end{equation*}
is harmonic in $\Omega$ with $z = 0$ on $\Gamma$ and $z\in H^{t}(\Omega)$ for any $t < 1 -\alpha $. Note that $0 <  1 -\alpha < 1/2$.\medskip

In \cite{AM} Theorem 8.3, we proved the following uniqueness result: 

\begin{theorem} [{\bf Uniqueness criterion in} $ H^{1/2}(\Omega)$] \label{unicityH1demi} Let $\Omega$ be a bounded Lipschitz domain of $\mathbb{R}^N$ with $N \geq 2$. 
  If $u\in H^{1/2}(\Omega)$ is harmonic and $u=0$ on $\Gamma$, then $u = 0$ in $\Omega$. 
  \end{theorem}

Let us now consider the case where $\Omega$ is a bounded Lipschitz polygonal domain, we will simply write polygon, whose angles are denoted by $\omega_1, \ldots, \omega_J$. Even if we have to reorder this family, we can assume that it is increasing: $\omega_1 \leq \ldots \leq \omega_J$. We set $\alpha_j = \frac{\pi}{\omega_j} \in\,  ]1/2, \infty[$.\medskip

Due to Theorem \ref{unicityH1demi}, we observe that $\mathscr{N}_s(\Omega) = \{0\}$ when $s = - 1/2$. \medskip

Recall that when $ s\in \mathbb{N}$, we have (see \cite{Gri1})
\begin{equation}\label{eqdim}
\mathrm{dim}\, \mathscr{N}_s(\Omega)  = \sum_{j = 1}^{j = J}\nu_s(\omega_j), \quad \mathrm{where}\quad \nu_s(\omega_j) = \mathrm{the\, largest \ integer} < \frac{1 + s}{\alpha_j}, \end{equation}
provided that for $s \geq 1$,
\begin{equation*}\label{conddim}
\forall j = 1\ldots, J, \quad \omega_j \notin \{\frac{1}{s+1}\pi, \ldots, \frac{s}{s+1}\pi\}. 
\end{equation*}

The case where the parameter $s$ is real is also interesting. Here, we will limit ourselves to cases where $-1/2 \leq s \leq 0$. We can easily verify that the equality \eqref{eqdim} holds also for such $s$. A brief explanation of this result can be given. The singularities involved in the structure of the kernel $\mathscr{N}_s(\Omega)$ are, as in the above example, of the type $(r^{-\alpha_j} - r^{\alpha_j})\mathrm{sin}(\alpha_j\theta)$, where here $1/2 < \alpha_j = \frac{\pi}{\omega_j}< 1$. However the function $\vert x \vert^{-\alpha_j}$ belongs to $H^t(B)$ for any $ t$ strictly less than $1 - \alpha_j $, where $B$ is the unit disk centered at origin.  Note the function  $\vert x \vert^{-\alpha_j}$ belongs also to the Besov space $[H^1(B), L^2(B)]_{\alpha_j, \infty}$. {Setting $t = - s$,  we get   $\frac{1 - s}{\alpha_j} \in \,  ]1, 2[$ and then  $\nu_{-s}(\omega_j) = 1 $. We conclude that
\begin{equation}\label{dimang}
\forall -1/2 \leq s \leq 0, \quad dim\,  \mathscr{N}_{s}(\Omega) =  Card\{j \in \{1, \ldots, J\}; \;  \omega_j > \pi\}.
\end{equation}

Since the family $(\alpha_j)_{1\leq j \leq J}$ is decreasing, then $
 \mathscr{N}_s(\Omega) = \{0\}$ iff $ s \leq \alpha_1 - 1$.  Note that the relation \eqref{eqdim} remains true if the parameter $s \geq -1/2$ is real or if $\Omega$ is a curvilinear polygonal open set. \medskip

\begin{remark}\label{remNoyau}\upshape  i) Note that $\mathscr{N}_0(\Omega) = \{0\}$ if and only if the polygon $\Omega$ is convex.\smallskip

ii) We have  $\mathscr{N}_{-1/2}(\Omega) = \{0\} $ for any polygon (and also for any bounded Lipschitz domain of $\mathbb{R}^N,$ with $N \geq 2$, see Theorem \ref{unicityH1demi}  above).\smallskip

iii) For any polygon $\Omega$, there exists $s_0(\Omega) \in \, ]0, 1/2[$ such that
\begin{equation*}
\mathscr{N}_{-s_0}(\Omega) = \{0\}\quad \mathrm{and}\quad \mathrm{for \, any}\; s < s_0(\Omega), \;   \mathscr{N}_{-s}(\Omega) \neq \{0\},
\end{equation*}
where  $s_0(\Omega) = 1 - \alpha_J$.\smallskip

iv) We conjecture that for any bounded Lipschitz domain $\Omega$, there exists $s_0(\Omega) \in \, ]0, 1/2[$ such that  
\begin{equation*}
\mathrm{for \, any}\; s > s_0(\Omega),\; \;   \mathscr{N}_{-s_0}(\Omega) = \{0\} \quad \quad \mathrm{and}\quad \quad \mathrm{for \, any}\; s < s_0(\Omega), \;   \mathscr{N}_{-s}(\Omega) \neq \{0\}.
\end{equation*}

 \medskip
 Let now $\Omega$ be a bounded Lipschitz polyhedral domain of $\mathbb{R}^3$,  we will simply write polyhedron. The situation is little bit different. We denote by $\Gamma_k$, $k = 1, \ldots,  K$ the faces of $\Omega$ and by $E_{jk}$ the edge between $\Gamma_j$ and $\Gamma_k$ when $\overline{\Gamma}_j$ and $\overline{\Gamma}_k$ intersect. The measure of the interior angle of the edge $E_{jk}$ is denoted by $\omega_{jk}$ and as in 2D we have the following property: 
$$
\mathrm{if}\;  \alpha_{jk} \geq 1 - s, \quad \mathrm{for \, any }\;  1 \leq j, k \leq K, \quad \mathrm{then }\quad \mathscr{N}_s(\Omega) = \{0\},
$$
where $\alpha_{jk} = \frac{\pi}{\omega_{jk}}$. However, if one of the numbers $ \alpha_{jk}$ is strictly less than $ 1 - s$, then $\mathrm{dim}\, \mathscr{N}_s(\Omega) = + \infty$. 
\end{remark}

\subsection{Interpolation of subspaces}

We recall in this subsection some interpolation results of subspaces. The first one is due to Ivanov and Kalton \cite{IK}. Let $(X_0, X_1)$ be a Banach couple with $X_0\cap X_1$ dense in $X_0$ and in $X_1$. Let $Y_0$ be a closed subspace of $X_0$ with codimension one. Setting for any $0 < \theta < 1$
$$
X_\theta = [X_0, X_1]_\theta \quad \mathrm{and}\quad Y_\theta = [Y_0, X_1]_\theta,
$$
we have the following result concerning the interpolation of subspaces:
 
\begin{theorem} [{\bf Ivanov-Kalton}]\label{thmIK} There exist two indices $0 \leq \sigma_0 \leq \sigma_1 \leq 1$ such that\\
i) If $0 < \theta < \sigma_0$, then the space $Y_\theta $ is a closed subspace of codimension one in the space $ X_\theta $.\\
ii) If $\sigma_0 \leq \theta \leq \sigma_1$, then the norm of $Y_\theta $ is not  equivalent to the norm of $X_\theta $.\\
iii) If $\sigma_1 < \theta < 1$, then $Y_\theta = X_\theta$ with equivalence of norms.
\end{theorem}

As specified in  \cite{IK}, the special case of a Hilbert space of Sobolev type connected with elliptical boundary value problem was studied in \cite{Lions}, with the well known case $X_0 = H^1(\Omega),\,  X_1 = L^2(\Omega)$ and   $Y_0 = H^1_0(\Omega)$, but where  $Y_0$ is here a closed subspace of codimension infinite in $X_0$. We recall that the corresponding critical values are $\sigma_0 = \sigma_1 = 1/2$ and $Y_{1/2} = H^{1/2}_{00}(\Omega)$.\medskip

The above theorem is generalized  by Asekritova, Cobos and Kruglyak \cite{ACK} when $Y_0$ is a closed subset of finite codimension $n$ in $X_0$:

\begin{theorem} [{\bf Asekritova-Cobos-Kruglyak}]\label{thmACK} There exist $2n$ indices satisfying $0 \leq \sigma_{0j} \leq \sigma_{1j} \leq 1$, $j = 1, \ldots, n$ and such that 
\begin{equation*}\label{CNSclose}
Y_\theta \; \; \mathrm{is\,\, closed\,\, in} \; \; X_\theta \quad \Longleftrightarrow \quad \theta \notin \displaystyle \bigcup_{j= 1}^{j= n} [\sigma_{0j} , \sigma_{1j} ].
\end{equation*}
Moreover, in that case if the cardinal 
\begin{equation*}\label{Cardinter}
\vert\{j \in \{1, \ldots, n \}; \; \theta < \sigma_{0j}\}\vert \quad \mathrm{is\,\, equal\,\, to}\;  k,
\end{equation*}
then the space $Y_\theta$ is a closed subspace of codimension $k$ in $X_\theta$.
\end{theorem}

\subsection{ $H^s(\Omega)$-regularity.}

Recall now the following result due to Grisvard (see Theorem 2.4 in \cite{Gri1} and  Theorem 2.1 in \cite{Gri2}). 

\begin{theorem} [{\bf Grisvard, $H^2$-Regularity}]\label{theoGrisvardPol} Let $\Omega$ be a polygonal domain of $\mathbb{R}^2$  or a polyhedral domain of $\mathbb{R}^3$. \\
i) The following inequality holds: there exists a constant $C(\Omega)$ such that
\begin{equation}\label{inegH2}
\forall v \in H^2(\Omega)\cap H^{1}_0(\Omega),  \quad \Vert v \Vert_{H^2(\Omega)} \leq C(\Omega) \Vert \Delta v \Vert_{L^2(\Omega)}. 
\end{equation}
ii) The following  operator is an isomorphism
\begin{equation}\label{DeltaIsoH3/2H1/2ortx}
\Delta : H^2(\Omega)\cap H^{1}_0(\Omega) \longrightarrow   [\mathscr{N}_0(\Omega)]^\bot ,
\end{equation}
where 
$$
  [\mathscr{N}_0(\Omega)]^\bot  = \{f \in L^2(\Omega); \; \forall \varphi\in \mathscr{N}_0(\Omega), \;  \int_\Omega f\varphi = 0\}.
  $$
iii) Codim$\mathrm{(}$Im $\Delta \mathrm{)}$  is finite in $2$D and infinite in $3D$.
\end{theorem}

In fact the constant $C(\Omega)$ depends only on the Poincar\'e constant of $\Omega$ (see \eqref{inegisoH2L2} in next remark point iii)).

\begin{remark}\label{remtrace}\upshape i) It is natural to ask the question of the meaning of traces for the functions $v \in L^2(\Omega)$ satisfying $\Delta v \in L^2(\Omega)$  when  $\Omega$ is a polygonal domain of $\mathbb{R}^2$  or a polyhedral domain of $\mathbb{R}^3$.  In \cite{Gri1} (see Lemma 3.2), the author define the trace of such function as an element of $[X(\Omega)]'$, where $X(\Omega)$ is the  space  described by $\frac{\partial \varphi}{\partial\textit{\textbf n}} $ when $\varphi$ browse through the space $H^{2}(\Omega)\cap H^{1}_0(\Omega)$.\medskip

\noindent ii) Unlike the case where the domain $\Omega$ is convex or is regular, of class $\mathscr{C}^{1,1}$ for example, the kernel $\mathscr{N}_0(\Omega)$ is not trivial. In the polygonal case,  as mentioned above (see also Subsection \ref{ssCharkernels}), it is of finite dimension and its dimension is equal to the number of vertices of the polygon whose corresponding interior angle is strictly greater than  $\pi$.\medskip

\noindent iii) To establish the fundamental inequality \eqref{inegH2}
 the author shows that for any $v\in  H^2(\Omega)\cap H^{1}_0(\Omega)$
\begin{equation*}
\int_\Omega \frac{\partial ^2 v}{\partial x^{2}}\, \frac{\partial ^2 v}{\partial y^{2}} =  \int_\Omega \big\vert \frac{\partial ^2 v}{\partial x \partial y}\big\vert^2. 
\end{equation*}
Therefore
\begin{equation*}
\Vert \Delta v \Vert^2_{L^{2}(\Omega)} = \Vert \frac{\partial ^2 v}{\partial x^{2}} \Vert^2_{L^2(\Omega)} + \Vert \frac{\partial ^2 v}{\partial y^{2}} \Vert^2_{L^2(\Omega)} + 2  \int_\Omega  \frac{\partial ^2 v}{\partial x^{2}}  \frac{\partial ^2 v}{\partial y^{2}}. 
\end{equation*}
Denoting by $\vert \cdot \vert_{H^{2}(\Omega)} $ the semi-norm $H^{2}(\Omega)$ and by $\nabla^2$ the Hessian matrix, we deduce that
\begin{equation}\label{inegH2DeltaL2H1}
\vert v \vert_{H^{2}(\Omega)} = \Vert \nabla^2 v \Vert_{L^{2}(\Omega)} =  \Vert \Delta v \Vert_{L^{2}(\Omega)}\quad \mathrm{and}\quad  \Vert v \Vert^2_{H^{2}(\Omega)}  \leq  \Vert \Delta v \Vert^2_{L^{2}(\Omega)} +  \Vert v \Vert^2_{H^{1}(\Omega)}.
\end{equation}
Besides,
\begin{equation*}
\int_\Omega \vert \nabla v\vert^2 = - \int_\Omega v \Delta v \leq C_P(\Omega) \Vert \nabla v \Vert_{L^{2}(\Omega)}\Vert \Delta v \Vert_{L^{2}(\Omega)},
\end{equation*}
where $C_P(\Omega)$ is the Poincar\'e constant.
And then
\begin{equation*}
\Vert \nabla v \Vert_{L^{2}(\Omega)} \leq C_P(\Omega)\Vert \Delta v \Vert_{L^{2}(\Omega)}.
\end{equation*}
Finally, we get the following estimate
\begin{equation}\label{inegisoH2L2}
\Vert v \Vert_{H^{2}(\Omega)}  \leq (1 +  C^2_P(\Omega))^{1/2} \Vert \Delta v \Vert_{L^{2}(\Omega)}.
\end{equation}
The polyedral case can be treated in the same way.
\end{remark}  

The solvability of problem 
$$
(\mathscr{L}_D^0)\ \ \ \  -\Delta u = f\quad \ \mbox{in}\ \Omega \quad
\mbox{and } \quad u = 0 \ \ \mbox{on }\Gamma,
$$ 
in a framework of fractional Sobolev spaces when $\Omega$ is a polygon or a polyhedron, and naturally also when   $\Omega$ is a general Lipschitz domain, has been studied by many authors.  Recall that when  
 $\Omega$ is a general Lipschitz domain, for any $1/2 < s < 3/2$, the operator
\begin{equation*}\label{isoDeltaLip}
 \Delta : H^{2- s}_0(\Omega) \longrightarrow H^{-s}(\Omega)
\end{equation*}
is an isomorphism. So, the question is: what happens in the case where $\Omega$ is a polygonal domain of $\mathbb{R}^2$  or a polyhedral domain of $\mathbb{R}^3$ and where $0 \leq s \leq 1/2$? A first answer can be given by the following corollary.

\begin{corollary}\label{firstcor} Let $\Omega$ be a polygonal domain of $\mathbb{R}^2$  or a polyhedral domain of $\mathbb{R}^3$. Then for any $0 < \theta < 1$, the operator
\begin{equation}\label{DeltaIsoH2-thetaH-thetabbis}
 \Delta : H^{2-\theta}(\Omega) \cap H^1_0(\Omega)\longrightarrow M_\theta(\Omega)
\end{equation}
is an isomorphism, where 
$$
M_\theta(\Omega) : = [(\mathscr{N}_0(\Omega))^\bot,  H^{-1}(\Omega)]_{\theta}
$$
 is continuously embedded in $  H^{-\theta}(\Omega)$ if $\theta \neq 1/2$ and in $[H^{1/2}_{00}(\Omega)]'$ if $\theta = 1/2$.
\end{corollary}

\begin{proof} i) First, let us recall that the operator
\begin{equation}\label{DeltaIsoH10H-1a}
\Delta : \; H^{1}_{0}(\Omega)\longrightarrow  H^{-1}(\Omega)
\end{equation}
is an isomorphism.  By a simple interpolation argument, thanks to \eqref{DeltaIsoH3/2H1/2ortx} and \eqref{DeltaIsoH10H-1a}, we deduce that  for any $0 < \theta < 1$, the operator
\begin{equation}\label{DeltaIsoH2-thetaH-thetab}
 \Delta : H^{2-\theta}(\Omega) \cap H^1_0(\Omega)\longrightarrow [[\mathscr{N}_0(\Omega)]^\bot,  H^{-1}(\Omega)]_{\theta}
\end{equation}
is an isomorphism. Moreover, using Fredholm alternative, the Laplace operator considered as operating from $H^{2-\theta}(\Omega) \cap H^1_0(\Omega)$  into $H^{-\theta}(\Omega)$  if $\theta \neq 1/2$, resp. into $[H^{1/2}_{00}(\Omega)]'$ if $\theta = 1/2$, verifies the relation:
\begin{equation}\label{MKer}
\overline{M_\theta(\Omega)} = [Ker (\Delta^\star)]^\bot = [\mathscr{N}_{-\theta}(\Omega)]^\bot,
\end{equation}
where 
$$
[\mathscr{N}_{-\theta}(\Omega)]^\bot = \{f\in H^{-\theta}(\Omega); \, \langle f, \varphi\rangle = 0, \, \forall \varphi \in \mathscr{N}_{-\theta}(\Omega)\}
$$
 if $\theta \neq 1/2$ (for $\theta = 1/2$, replace $H^{-\theta}(\Omega)$ by $[H^{1/2}_{00}(\Omega)]'$). Note that $\mathscr{N}_{-\theta}(\Omega) = \{0\}$ for $\theta \geq 1/2$.\medskip
 
 ii) Let us prove that the operator \eqref{DeltaIsoH2-thetaH-thetab} is an isomorphism.  Indeed,  setting $S= \Delta^{-1}$ we know from \eqref{DeltaIsoH3/2H1/2ortx} and \eqref{DeltaIsoH10H-1a} that  the operators $S : \; H^{-1}(\Omega))\longrightarrow  H^{1}_{0}(\Omega)$ and $S : \; [\mathscr{N}_0(\Omega)]^\bot\longrightarrow  H^2(\Omega)\cap H^{1}_0(\Omega)$ are continuous. So by interpolation, we  also  have the continuity of the operator 
\begin{equation}\label{opS}
 S : \; M_\theta(\Omega) \longrightarrow  H^{2-\theta}(\Omega) \cap H^1_0(\Omega).
 \end{equation}
 Clearly this last operator is injective. Let us prove its surjectivity. Before that recall that the density of a Banach $X$ in a Banach $Y$ implies the density of $X$ in the interpolate space $[X, Y]_\theta$ (see \cite{Lions}).  Since  $ H^{2}(\Omega) \cap H^1_0(\Omega)$ is dense in  $H^1_0(\Omega)$, we deduce that $[\mathscr{N}_0(\Omega)]^\bot$ is also dense in $H^{-1}(\Omega)$. Applying this property with $X = [\mathscr{N}_0(\Omega)]^\bot$  and $Y =  H^{-1}(\Omega)$, we deduce the density of $[\mathscr{N}_0(\Omega)]^\bot$ in $ M_\theta(\Omega)$. Given now $f\in M_\theta(\Omega)$, there exists a sequence $f_j \in [\mathscr{N}_0(\Omega)]^\bot$  such that $f_j \rightarrow f$ in $M_\theta(\Omega)$. Setting $u_j = Sf_j \in  H^2(\Omega)\cap H^{1}_0(\Omega)$,  we know that 
$$
\Vert Sf_j \Vert_{H^{2-\theta}(\Omega)}\leq  C\Vert f_j \Vert_{M_\theta(\Omega)}.
$$ 
Consequently $u_j \rightarrow u$ in $ H^{2-\theta}(\Omega)$ with $\Delta u = f$ in $\Omega$ and $u = 0$ on $\Gamma$. So the operator \eqref{opS} is surjective and then is an isomorphism.
\end{proof}

\begin{remark}\upshape The reader's attention is drawn here to the interpolation argument used above. The fact that  the operators  \eqref{DeltaIsoH10H-1a} and \eqref{DeltaIsoH3/2H1/2ortx} are isomorphisms does not necessarily mean that the interpolated operator is as well. However, the invertibility holds in our case thanks to the density of $H^2(\Omega)\cap H^{1}_{0}(\Omega) $ in $H^{1}_{0}(\Omega)$.
We can find some counterexamples in \cite{FJ} with invertible operator on $ L^p(\mathbb{R})$ and $ L^q(\mathbb{R})$ but not on $ L^r(\mathbb{R})$ for some $r$ between $p$ and $q$. 
\end{remark}
The difficulty now lies in determining the interpolated space $M_\theta(\Omega)$. However the optimal regularity  in the case of polygonal or polyhedral domains is in fact better (see Remark \ref{rem4}, Point i)  below) and can be expressed as follows:

\begin{theorem} [{\bf $H^s$- Regularity I}]\label{theoHsregul}  Let $\Omega$ be a polygonal domain of $\mathbb{R}^2$  or a polyhedral domain of $\mathbb{R}^3$ that we assume non convex. We denote by $\omega^\star$ the measure of the largest interior angle of $\Omega$ and we set $\alpha^\star = \pi/\omega^\star$. Then\\  
i) for  any $\theta \in \, ] 1 - \alpha^\star, 1[$ with $ \theta \neq 1/2$, the operator
\begin{equation}\label{DeltaIsoH2-thetaH-thetaterz}
 \Delta : H^{2-\theta}(\Omega) \cap H^1_0(\Omega)\longrightarrow {H}^{-\theta}(\Omega)
\end{equation}
is an isomorphism and\smallskip

ii) for  $ \theta = 1/2$, the operator
\begin{equation}\label{DeltaIsoH2-thetaH-1demibisz}
 \Delta : H^{3/2}_0(\Omega)\longrightarrow  [{H}^{1/2}_{00}(\Omega)]'
\end{equation}
is also an isomorphism (see Proposition \ref{corGrisvardPol} for more information). \smallskip

iii) We have the following characterizations:
\begin{equation}\label{carainterb}
  M_{\theta}(\Omega) =  \begin{cases} H^{-\theta}(\Omega) \quad \textrm{ if }\; \theta \in \, ] 1 - \alpha^\star,\;  1[\quad \mathrm{with}\;  \; \; \theta \neq 1/2\\
 [H^{1/2}_{00}(\Omega)]' \quad\textrm{ if }\; \theta = 1/2,
 \end{cases}
 \end{equation}
 with equivalent norms.
\end{theorem}

\begin{proof}  For Point i) and Point ii), see for instance \cite{BBX} and also \cite{D} Theorem 18.13.\smallskip

Using \eqref{DeltaIsoH2-thetaH-thetaterz} and \eqref{DeltaIsoH2-thetaH-1demibisz}, by identification we deduce from \eqref{DeltaIsoH2-thetaH-thetabbis} the  characterization \eqref{carainterb}.
\end{proof}

\begin{remark} \label{rem4}\upshape i) Note that since $\Omega$ is assumed non convex, then $0 < 1 -\alpha^\star < 1/2$.  
When $\alpha^\star$ is close to 1, the domain $\Omega$ is close to be convex and we are near the $H^2$-regularity. Conversely, when $\alpha^\star$ is near 1/2, the  domain $\Omega$ is close to a cracked domain and the expected regularity in this case is better than $ H^{3/2}$. That means that for any nonconvex polygonal  or  polyhedral domain, there exists $\varepsilon = \varepsilon(\omega^\star)\in \, ]0, 1/2[$ depending on  $\omega^\star$ (in fact, $\varepsilon(\omega^\star) = \alpha^\star - 1/2$) such that for any $ 0 < s < \varepsilon$ and any $f \in H^{-\frac{1}{2} + s}(\Omega)$ the  $H^1_0(\Omega)$ solution of Problem $(\mathscr{L}_D^0)$ belongs to $ H^{3/2 + s}(\Omega)$.\smallskip

ii) A natural question to ask concerns the limit case $ \theta = 1 - \alpha^\star$, where the regularity $H^{1 + \alpha^\star}(\Omega)$ is in general not achieved for 
RHS $f$ in $H^{-1 + \alpha^\star}(\Omega)$. However it is attained if we suppose $f\in \bigcap_{s > \alpha^\star}H^{-1 + s}(\Omega)$  (see \cite{BBX}). \smallskip

iii) What happens if $0< \theta < 1 - \alpha^\star$. This question will be examined  a little further on (see Theorem \ref{regpolthetasm}).
\end{remark}

The following proposition is important because it specifies the dependence of one of the constants involved in the equivalence given in Theorem \ref{theoHsregul} Point iii).

\begin{e-proposition}\label{corGrisvardPol}   Let $\Omega$ be a polygonal domain of $\mathbb{R}^2$  or a polyhedral domain of $\mathbb{R}^3$. Then 
\begin{equation*}\label{inegPolyH3demi}
\forall v \in H^{2-\theta}(\Omega)\cap H^1_0(\Omega),  \quad \Vert v \Vert_{H^{2-\theta}(\Omega)} \leq C(\Omega) \Vert \Delta v \Vert_{M_{\theta}(\Omega)},
\end{equation*}
where the contant $C(\Omega)$ above depends only on $\theta$,  on the Poincar\'e constant and the Lipschitz character  of  $\Omega$.
\end{e-proposition}

\begin{proof} Let us begin by introducing the following operators: for any $i, j = 1, 2$, we set 
$$
K_{ij} = \frac{\partial^2}{\partial  x_i\partial x_j}\quad \mathrm{and}\quad   L_{ij} = K_{ij} \Delta^{-1},
$$ 
where $\Delta^{-1}$ denotes the inverse of Laplacian as in the proof of Corollary \ref{firstcor}. Using  \eqref{inegH2DeltaL2H1}  we get for any $i, j = 1$ or $2$ the following
\begin{equation}\label{inegLijgrdade2L2}
\forall f\in [\mathscr{N}_0(\Omega)]^\bot,\quad  \Vert  L_{ij}f \Vert_{L^2(\Omega)}\leq \Vert  f \Vert_{L^2(\Omega)}.
\end{equation} 
We know that
\begin{equation}\label{ineggrad2H-1}
\forall v \in {H^{1}_0(\Omega)}, \quad  \Vert \nabla^2 v \Vert_{H^{-1}((\Omega)} \leq \Vert \nabla v \Vert_{L^{2}((\Omega)} \leq  C_P(\Omega)\Vert \Delta v \Vert_{H^{-1}(\Omega)}.
\end{equation} 
From inequality \eqref{ineggrad2H-1} and  isomorphism \eqref{DeltaIsoH10H-1a} we have also
\begin{equation}\label{inegLijgrdade2H-1}
\forall f\in H^{-1}(\Omega),\quad  \Vert  L_{ij}f \Vert_{H^{-1}(\Omega)}\leq  C_P(\Omega)\Vert  f \Vert_{H^{-1}(\Omega)}.
\end{equation} 
Recall that if $T$ is a linear and continuous operator from $E_0$ into $F_0$ and  from $E_1$ into $F_1$, where $E_j$ and $F_j$ are Banach spaces, then for any $0 < \theta < 1$ the linear operator $T$ is  also continuous  from $E_\theta = [E_0, E_1]_\theta $ into  $F_\theta = [F_0, F_1]_\theta $  and we have the following interpolation inequality: for any $v\in E_\theta$, 
\begin{equation}\label{inegabstinter}
\Vert Tv \Vert_{F_{\theta}} \leq \Vert T \Vert_{\mathscr{L}(E_0; F_0)}^{1 - \theta}  \Vert  T \Vert_{\mathscr{L}(E_1; F_1)}^{\theta}\Vert v \Vert_{E_{\theta}},
\end{equation} 
(see Adams \cite{Ada} page 222, Berg-Lofstr$\mathrm{\ddot{o}}$m \cite{BL} Theorem 4.1.2 and Triebel \cite{Tri} Remark 3 page 63).\medskip

We deduce from \eqref{inegabstinter}, \eqref{inegLijgrdade2H-1}, \eqref{inegLijgrdade2L2}, \eqref{carainterb} and \eqref{DeltaIsoH2-thetaH-thetabbis} the following inequalities: for any $f\in M_\theta(\Omega)$ 
\begin{equation*}\label{inegLijgrdade2H-1int}
 \Vert  L_{ij}f \Vert_{H^{-\theta}(\Omega)}\leq  C_P(\Omega)\Vert  f \Vert_{M_\theta(\Omega)}, \; \mathrm{if}\; \theta \neq 1/2\quad \mathrm{and}\quad    \Vert  L_{ij}f \Vert_{[H^{1/2}_{00}(\Omega)]'}\leq  C_P(\Omega)\Vert  f \Vert_{M_{1/2}(\Omega)}\;   \mathrm{if}\; \theta = 1/2.
\end{equation*}
Therefore,  we have  the following estimate: for any $v \in H^{2-\theta}(\Omega)\cap H^1_0(\Omega),$
\begin{equation}\label{inegPolyH3demiba}
\Vert \frac{\partial^2 v}{\partial  x_i\partial x_j} \Vert_{H^{-\theta}(\Omega)} \leq C_P(\Omega) \Vert \Delta v \Vert_{M_{\theta}(\Omega)},  \; \mathrm{if}\; \theta \neq 1/2,\; \Vert \frac{\partial^2 v}{\partial  x_i\partial x_j} \Vert_{[H^{1/2}_{00}(\Omega)]'} \leq C_P(\Omega) \Vert \Delta v \Vert_{M_{1/2}(\Omega)}\;   \mathrm{if}\; \theta = 1/2.
\end{equation}
Moreover, using  Corollary 3.4 in \cite{AM} and \eqref{inegPolyH3demiba} we get for any $v \in H^{2-\theta}(\Omega)\cap H^1_0(\Omega),$
\begin{equation*}
 \Vert  v \Vert_{H^{2-\theta}(\Omega)} \leq C_{P, L} (\Omega)\Vert \nabla^2 v \Vert_{H^{-\theta}(\Omega)} \leq C_{P, L} (\Omega)\Vert \Delta v \Vert_{M_{\theta}(\Omega)} \quad \mathrm{if}\; \theta \neq 1/2
\end{equation*}
and where we replace $H^{2-\theta}(\Omega)$ by $H^{3/2}(\Omega)$ and $H^{-\theta}(\Omega)$ by $[H^{1/2}_{00}(\Omega)]'$ if $\theta = 1/2$. That concludes the proof.\end{proof}

We are now in a position to extend Theorem \ref{theoHsregul}  to the case where $0< \theta \leq 1 - \alpha^\star$. 
We begin by considering the case more simple  where the domain $\Omega$ is a polygon with only one angle having a measure $\omega^\star$ larger than $\pi$ and of vertex $A$. We know from \eqref{dimang} that the kernel $ \mathscr{N}_{-\theta}(\Omega) $ is of dimension $1$, say $\mathscr{N}_{-\theta}(\Omega) = \langle z\rangle $. 

\begin{theorem} [{\bf $H^s$- Regularity II}] \label{regpolthetasm} Let $\Omega$ be a polygon with only one angle having a measure $\omega^\star$ larger than $\pi$. Then,\\
i) for  any $0 \leq \theta < 1 - \alpha^\star, $ the operator
\begin{equation}\label{DeltaIsoH2-thetaH-thetabis}
 \Delta : H^{2- \theta}(\Omega) \cap H^1_0(\Omega)\longrightarrow \langle z \rangle^\bot
\end{equation}
is an isomorphism, where $
 \langle z \rangle^\bot = \{\varphi \in H^{-\theta}(\Omega); \; \langle \varphi, z\rangle = 0\},$\smallskip

ii) for $\theta = 1 - \alpha^\star, $ the operator
\begin{equation*}\label{CC}
 \Delta : H^{1 + \alpha^\star}(\Omega) \cap H^1_0(\Omega)\longrightarrow M_{1 - \alpha^\star}(\Omega)
\end{equation*}
is an isomorphism. Moreover  
$$
\bigcap_{r < 1 - \alpha^\star}H^{-r}(\Omega) \hookrightarrow M_{1 - \alpha^\star}(\Omega)\hookrightarrow H^{-1 + \alpha^\star}(\Omega),
$$
where the topology of $ M_{1 - \alpha^\star}(\Omega)$ is finer than that of $ H^{-1 + \alpha^\star}(\Omega)$.
\end{theorem}

\begin{proof} First, let's remember that for any $0 < \theta < 1$ we have the  isomorphism \eqref{DeltaIsoH2-thetaH-thetabbis}.

Our goal is to prove that  $\sigma_0 = \sigma_1= 1 - \alpha^\star$, with the same notations as in Theorem \ref{thmIK}. The above results will then be a consequence of Theorem \ref{theoHsregul}, Corollary \ref{firstcor} and the characterization of Kernel $\mathscr{N}_{-\theta}(\Omega)$, which is of dimension 1 for all $0\leq \theta <  1 - \alpha^\star$ and reduced to $\{0\}$ when  $ \theta >  1 - \alpha^\star$.\medskip

First, thanks to Theorem \ref{theoHsregul},  we have $\sigma_1 \leq  1 - \alpha^\star$. Second, using Point ii) of Theorem \ref{theoGrisvardPol} and Point ii) of Theorem  \ref{thmIK} we deduce that $\sigma_0 > 0$. In fact, the value of $\sigma_0 > 0$ is directly related to the function $z$ of the kernel $\mathscr{N}_{-\theta}(\Omega)$, when $0\leq \theta <  1 - \alpha^\star$, which belongs to the Besov space $[H^1(\Omega), L^2(\Omega)]_{\alpha^{\star}, \infty}$, but not to $H^{1 - \alpha^\star}(\Omega)$. So the only possible value of  $\sigma_0$ is $ 1 - \alpha^\star$ and then $\sigma_0 = \sigma_1= 1 - \alpha^\star$.  Using then the isomorphism \eqref {DeltaIsoH2-thetaH-thetab}, the identity \eqref{MKer} and  Theorem \ref{thmIK}, we conclude that:
$$
M_\theta(\Omega) = \langle z \rangle^\bot\quad \mathrm{for\, any}\quad 0 \leq \theta < 1 - \alpha^\star.
$$
Moreover the norm of $M_\theta(\Omega)$ is equivalent to the norm of $H^{-\theta}(\Omega)$, where the corresponding constants involved in this equivalence depend on $ \alpha^\star$.\medskip

Now, concerning the critical value $1 - \alpha^\star$ and according to Theorem \ref{thmIK}, Point ii), the norm of the space $M_{1 - \alpha^\star}(\Omega)$ is not equivalent to the norm of  $ H^{-1 + \alpha^\star}(\Omega)$. So all the properties stated in the theorem above take place.\end{proof}

\begin{remark}\upshape \noindent  i) In \cite{BBX}, Theorem 4.1,  the authors proved that for any $f$ belonging to some subspace of the Besov space $[H^1(\Omega), L^2(\Omega)]_{\alpha^{\star}, \infty}$, where $\Omega$ satisfies the assumptions of the above theorem, the $ H^1_0(\Omega)$ function $u$ satisfying  $\Delta u = f$ in $\Omega$ is in fact in the Besov space $[H^2(\Omega), H^1(\Omega)]_{\alpha^{\star}, \infty}$ which contains the space $H^{1 + \alpha^\star}(\Omega) \cap H^1_0(\Omega)$. That means that our result in Point ii) above is little bit better.
\smallskip

\noindent ii) For $f$ given in ${H}^{-s}(\Omega)$ with $0 \leq s < 1 - \alpha^\star$, let $u$ be the solution $H^1_0(\Omega)$ of the problem $(\mathscr{L}_D^0)$.  In fact $u\in  H^{3/2}_0(\Omega)$ and we know that $u$ is more regular outside a neighborhood $V$ of the vertex $A$: precizely $u \in H^{2- s}(\Omega\setminus V) $. It is convenient to introduce polar coordinates $(r, \theta)$ centered at $A$ such that the two sides of the angle correspond to $\theta = 0$ and $\theta = \omega^\star$. Let us introduce the following function:
$$
S(r, \theta) = \eta(r) r^{\alpha^\star} sin(\alpha^\star \theta) 
$$
where $\alpha^\star = \pi/\omega^\star $ and $\eta$ is a regular cut-off function defined on $\mathbb{R}^+$, equal to 1 near zero, equal to zero in some interval $ [a, \infty[$, with some small $a > 0$. Observe that $S \in H^t(\Omega)$ for any $t < 1+ \alpha^\star$ and  $\Delta S \in H^t(\Omega)$ for any $t <  \alpha^\star$. The function $w = u - \lambda S$, with $\lambda$ constant to be determined later, satisfies $\Delta w  \in {H}^{-s}(\Omega)$.  If $\langle f, z \rangle = 0$, we deduce from \eqref{DeltaIsoH2-thetaH-thetabis} that $u\in H^{2-s}(\Omega)$. If  $\langle f, z \rangle \neq 0$, we choose $\lambda =  \langle \Delta S, z \rangle/\langle f, z \rangle$. So $\langle \Delta w, z \rangle = 0$ and then  $w\in H^{2-s}(\Omega)$. That means that $u - \lambda S$ belongs to  $H^{2-s}(\Omega)$.\smallskip

\noindent iii) In \cite{J-K}, the authors recall that for any $s > 3/2$ there is a Lipschitz domain $\Omega$ and $f\in \mathscr{C}^\infty(\overline{\Omega})$ such that the solution $u$ to the inhomogeneous Dirichlet problem $(\mathscr{L}_D)$  does not belong to $H^{s}(\Omega)$. Point i) above shows that the compatibility condition is the hypothesis on $f$ that allows us to obtain the expected regularity $H^{s}(\Omega)$.

\end{remark}
We will  now consider  the case  where the domain $\Omega$ is a polygon having  $n$ angles  $\omega_1, \ldots, \omega_n$, with $n\geq 2$, larger than $\pi$. We suppose that $\omega_1 \leq \ldots \leq \omega_n$. We know from \eqref{eqdim} that the kernel $ \mathscr{N}_{-\theta}(\Omega) $ is of dimension $n$, say $\mathscr{N}_{-\theta}(\Omega) = \langle z_1, \ldots, z_n\rangle $.

\begin{theorem} [{\bf $H^s$- Regularity III}] \label{regpolthetasm+}  Let $\Omega$ be a polygon.  We denote by  $\omega_1, \ldots, \omega_n$, with $n\geq 2$, all angles larger than $\pi$ and suppose $\omega_1 \leq \ldots \leq \omega_n$.  Setting  $\alpha_k = \pi/\omega_k$ for $k = 1, \ldots, n$ and by convention $\alpha_0 = 1$, then\\
i) for  any $\theta \in \, ] 1 - \alpha_n, 1[$ with $ \theta \neq 1/2$, the operator $
 \Delta : H^{2-\theta}(\Omega) \cap H^1_0(\Omega)\longrightarrow {H}^{-\theta}(\Omega)$ is an isomorphism,\\
ii) for  $ \theta = 1/2$, the operator $\Delta : H^{3/2}_0(\Omega)\longrightarrow  [{H}^{1/2}_{00}(\Omega)]'$ is an isomorphism,\\
iii) for any fixed $k = 0, \ldots, n-1$, $\theta \in \, ] 1 - \alpha_k, 1 - \alpha_{k+1}[$ and $f\in H^{-\theta}(\Omega)$ satisfying the following compatibility condition
\begin{equation*}
\forall \varphi \in \langle z_{k+1}, \ldots, z_{n}\rangle, \quad \langle f, \varphi\rangle = 0,
\end{equation*}
Problem  $(\mathscr{L}_D^0)$ has a unique solution $u \in H^{2-\theta}(\Omega)$.
\end{theorem}

We do not provide proof of this theorem. It is similar to that of Theorem \ref{regpolthetasm}  and involves Theorem \ref{thmACK}.

\begin{remark}\upshape i) For critical values $\theta = 1 - \alpha_k $, with $k = 1, \ldots, n$, the operator $
 \Delta : H^{1 + \alpha_k}(\Omega) \cap H^1_0(\Omega)\longrightarrow M_{1 - \alpha_k}(\Omega)$ 
is an isomorphism. Moreover  
$$
\bigcap_{r < 1 - \alpha_k}H^{-r}(\Omega) \hookrightarrow M_{1 - \alpha_k}(\Omega)\hookrightarrow H^{-1 + \alpha_k}(\Omega),$$
where the topology of $ M_{1 - \alpha_k}(\Omega)$ is finer than that of $ H^{-1 + \alpha_k}(\Omega)$.\smallskip

\noindent{ii)} For $\theta \in \, ] 1 - \alpha_n, 1[$ we have the following estimate: there exists a constant $C$ depending on $\alpha_n$ such that
\begin{equation}\label{equiNormPolyn}
\forall  f\in M_{\theta}(\Omega), \quad \Vert f \Vert_{M_{\theta}(\Omega)} \leq C \Vert f \Vert_{H^{-\theta}(\Omega)}.
\end{equation}
And for  $\theta \in \, ] 1 - \alpha_k, 1 - \alpha_{k+1}[$, with  $k = 0, \ldots, n-1$ we have the following estimate: there exists a constant $C$ depending on $\alpha_k$ and $\alpha_{k+1}$ such that inequality \eqref{equiNormPolyn} holds. So, as in the proof of Proposition \ref{corGrisvardPol}, we deduce the following inequalities: if $\theta \in \,  ] 1 - \alpha_n, 1[$, respectively  $\theta \in \, ] 1 - \alpha_k, 1 - \alpha_{k+1}[$, for some  $k = 0, \ldots, n-1$, then
\begin{equation*}
\forall v\in H^{2-\theta}(\Omega) \cap H^1_0(\Omega), \quad \Vert v \Vert_{H^{2-\theta}(\Omega)} \leq C \Vert \Delta v \Vert_{H^{-\theta}(\Omega)}, 
\end{equation*}
where the constant $C$ depends on the Poincar\'e constant of $\Omega$, on  the Lipschitz constant of $\Omega$  and on $ \alpha_n$, respectively on  $\alpha_k$ and $  \alpha_{k+1}$.

\end{remark}

 \section{Area integral estimate. Counterexample}\label{counterexample} 
 
The following proposition shows that  the inequalities \eqref{inegaltraceL2Gamma1} and \eqref{ineg1} cannot be valid in their current form.

\begin{e-proposition} [{\bf Counterexample for Inequalities \eqref{inegaltraceL2Gamma1} and \eqref{ineg1}}]  \label{ConterexampleH1demitrace1b} For any  $\varepsilon > 0$, there exist a Lipschitz domain $\Omega_\varepsilon \subset \mathbb{R}^2$ and a harmonic function $w_\varepsilon \in H^{3/2}(\Omega_\varepsilon)$ (with $\sqrt \varrho_\varepsilon\, \nabla^2 w_\varepsilon  \in L^2(\Omega_\varepsilon))$ such that the following family 
$$
(\Vert \varrho_\varepsilon \nabla^2 w_\varepsilon \Vert_{L^{2}(\Omega_\varepsilon)} + \Vert w_\varepsilon \Vert_{H^{3/2}(\Omega_\varepsilon)})_\varepsilon,
$$
is bounded with respect $\varepsilon$ and
\begin{equation*}
\Vert w_\varepsilon \Vert_{H^1(\Gamma_\varepsilon)} \rightarrow \infty \quad as \; \varepsilon\rightarrow  0.
\end{equation*}

\end{e-proposition}

\begin{proof} Recall first that if $u\in L^2(\Omega) $ is  harmonic, then (see Section 3 in \cite{AM} or Theorem 4.2 in \cite{J-K})
$$
u\in H^{1/2}(\Omega) \Longleftrightarrow \sqrt \varrho \, \nabla u \in L^2(\Omega)\; \mathrm{and}\; u\in H^{3/2}(\Omega) \Longleftrightarrow \sqrt \varrho \, \nabla^2 u \in L^2(\Omega).
$$

\medskip

\noindent{\bf Step 1.}  We suppose now that  $\Omega= \,  ]0, 1/2[\,  \times\,  ]0, 1/2[ $ and for any $\varepsilon > 0$ and close to 0, we define, as in \cite{J-K}, the following open set:
$$
\Omega_\varepsilon = \{(x, y) \in \mathbb{R}^2; \;  0 < x < 1/2\;\; \mathrm{and}\; \;  \varepsilon \lambda (x/\varepsilon)< y < 1/2\},
$$
where $
\lambda (x) = x\; \mathrm{ for} \;  0\leq x \leq 1,  \quad  \lambda (x) = 2- x \; \mathrm{ for } \;  1\leq x \leq 2 \quad \mathrm{ and } \; \; \lambda(x + 2) =\lambda (x).$ We set
$$
\Gamma_\varepsilon = \{(x, y) \in \mathbb{R}^2; \;  y = \varepsilon \lambda (x/\varepsilon), \; \mathrm{with}\;  0 < x < 1/2\},
$$
which is just a part of the boundary of $\Omega$.\medskip

\noindent{\bf Step 2.} Let us consider the following function (see also Example 2.1 given in Chapter 6 of the book of Ne\v{c}as) which depends only on the variable $y$:
$$
v(x, y) = \int_0^{y} ds\int_{1/2}^s \frac{dt}{t\, \ell n\, t}.
$$
It is easy to verify that $\sqrt{\varrho}\, \nabla^2 v \in \textbf{\textit L}^2(\Omega)$ where $\varrho$  is the distance to the boundary of $\Omega$. So we have also $v\in H^{3/2}(\Omega)$. Moreover, 
$$
\partial_\tau v (x, x) = \frac{\sqrt 2}{2} \int_x^{1/2} \frac{dt}{t\, \vert  \ell n\, t\vert} \quad \mathrm{ for } \; \; 0 <  x < \varepsilon,
$$
and 
$$
\partial_\tau v (x, 2\varepsilon - x) = - \frac{\sqrt 2}{2}  \int_{2\varepsilon - x}^{1/2} \frac{dt}{t\, \vert  \ell n\, t\vert} \quad \mathrm{ for } \; \; \varepsilon <  x < 2\varepsilon.
$$
Observe that 
$$
 \int_x^{1/2} \frac{dt}{t\, \vert  \ell n\, t\vert} = \ell n(- \ell n\, x) -  \ell n( \ell n\, 2).
 $$
So we have
\begin{equation*}
\begin{array}{rl}
\displaystyle\int_{\varepsilon}^{2\varepsilon}\left(\int_{2\varepsilon - x}^{1/2}\frac{dt}{t\, \vert  \ell n\, t\vert}\right)^2 dx = & \displaystyle\int_{\varepsilon}^{2\varepsilon} [\ell n(- \ell n\, (2\varepsilon - x)) -  \ell n( \ell n\, 2)]^2 dx\\
= &\displaystyle \int_{0}^{\varepsilon} [\ell n(- \ell n\, x) -  \ell n( \ell n\, 2)]^2 dx.
\end{array}
\end{equation*}
In addition as the function $\partial_\tau v$ is periodic with the period $ 2\varepsilon$ we get
$$
\Vert\partial_\tau v\Vert_{L^2(\Gamma_\varepsilon)}^2 =  (1/4\varepsilon)\int_0^{\varepsilon}\left(\int_x^{1/2}\frac{dt}{t\, \vert  \ell n\, t\vert}\right)^2 dx, 
$$
where we have chosen $\varepsilon = 1/4k$, with $k\in \mathbb{N}^\star$. 
\medskip

\noindent{\bf Step 3.}  We will give an estimate of the integral:
$$
I_\varepsilon = \int_0^{\varepsilon}\left(\int_x^{1/2}\frac{dt}{t\, \vert  \ell n\, t\vert}\right)^2 dx. 
$$
Setting $s =  \ell n \, t$, we get 
$$
\int_x^{1/2}\frac{dt}{t\, \vert  \ell n\, t\vert} = -  \int_{ \ell n\, x}^{-\ell n\, 2}\frac{ds}{s} = \ell n(- \ell n\, x) -  \ell n( \ell n\, 2).
$$
So $I_\varepsilon  = I_{1\varepsilon} - I_{2\varepsilon} + I_{3\varepsilon}$, where
$$
 I_{1\varepsilon} =  \int_0^{\varepsilon}  [\ell n(- \ell n\, x)]^2dx \quad I_{2\varepsilon} = 2 \ell n( \ell n\, 2)\int_0^{\varepsilon} \ell n(- \ell n\, x)dx
 $$
and
$$
 I_{3\varepsilon} =  \int_0^{\varepsilon}  [\ell n( \ell n\, 2)]^2dx.
$$
Since the function $\ell n(- \ell n\, x)$ is nondecreasing in the interval $]0, \varepsilon]$ and  $I_\varepsilon  \geq I_{1\varepsilon} - I_{2\varepsilon} $, we deduce easily the estimate for $\varepsilon$ close to 0:
\begin{equation*}
 I_{\varepsilon} \geq  \frac{\varepsilon}{2} \ell n(-\ell n\, \varepsilon)]^2
\end{equation*}
So, we get
$$
 \frac{1}{2\sqrt 2} [\ell n(-\ell n\, \varepsilon)]  \leq \Vert\partial_\tau v\Vert_{L^2(\Gamma_\varepsilon)}  \longrightarrow \infty \quad \mathrm{as} \quad \varepsilon \rightarrow 0.
$$

\noindent{\bf Step 4.} Now, using Theorem \ref{theoHsregul}  and Proposition \ref{corGrisvardPol} in the polygon $\Omega_\varepsilon$, there exits a unique solution $ H^{3/2}_0(\Omega_\varepsilon)$ satisfying $\Delta u_\varepsilon = \Delta v$ in $\Omega_\varepsilon$ with the estimate 
$$
\Vert   u_\varepsilon  \Vert_{H^{3/2}(\Omega_\varepsilon)}\leq C(\Omega_\varepsilon) \Vert \Delta v \Vert_{M_{1/2}(\Omega_\varepsilon)},
$$ 
where $ C(\Omega_\varepsilon)$ depends only on the Poincar\'e constant and on the Lipschitz character of $\Omega_\varepsilon$, which are both  bounded with respect $\varepsilon$. So we have  $  C(\Omega_\varepsilon) \leq  C(\Omega)$. Clearly the sequence $(\Vert \Delta v \Vert_{M_{1/2}(\Omega_\varepsilon)})_\varepsilon$, where we recall that $k = 1/(4\varepsilon)$, is increasing when $\varepsilon$ tends to zero. And since $\Omega$ is convex, then 
$$
M_{1/2}(\Omega) =  [L^2(\Omega), H^{-1}(\Omega)]_{1/2} = [H^{1/2}_{00}(\Omega)]',
$$
and $\lim_{\varepsilon\rightarrow 0} \Vert \Delta v \Vert_{M_{1/2}(\Omega_\varepsilon)} = \Vert \Delta v \Vert_{[H^{1/2}_{00}(\Omega)]'}.$ We then deduce that 
$$
\Vert   u_\varepsilon  \Vert_{H^{3/2}(\Omega_\varepsilon)}\leq C(\Omega) \Vert \Delta v \Vert_{[H^{1/2}_{00}(\Omega)]'}\leq C(\Omega) \Vert  v \Vert_{H^{3/2}(\Omega)}.
$$ 
So the harmonic function in $\Omega_\varepsilon$ defined by $w_\varepsilon = v - u_\varepsilon $ belongs to $H^{3/2}(\Omega_\varepsilon)$ and satisfies 
$$
 \Vert   w_\varepsilon  \Vert_{H^{3/2}(\Omega_\varepsilon)}\leq C(\Omega) \Vert  v \Vert_{H^{3/2}(\Omega)}.
 $$

Then, as $u_\varepsilon  = 0$ on $\partial\Omega_\varepsilon$, we have $\partial_\tau w_\varepsilon = \partial_\tau v$ on $\Gamma_\varepsilon$ and $w_\varepsilon$ would verify the following equality: $
\Vert\partial_\tau w_\varepsilon\Vert_{L^2(\Gamma_\varepsilon)} = \Vert\partial_\tau v\Vert_{L^2(\Gamma_\varepsilon)},$ which tends to infinity as $\varepsilon$ tends to zero.\end{proof} \medskip

Thus, there can be no estimate for harmonic function $u$, on the norm $\Vert u \Vert_{H^1(\Gamma)}$ by the norm  $\Vert u \Vert_{ H^{3/2}(\Omega)}$  dependent only on the Lipschitz constant of the considered  domain $\Omega$. More precisely, the following inequality
\begin{equation*}
\Vert u \Vert_{H^1(\Gamma)} \leq  C(\Omega) \Vert u \Vert_{ H^{3/2}(\Omega)}
\end{equation*}
cannot be satisfied in general, unlike when the domain is of class $\mathscr{C}^{1, 1}$, as we will see in the next section. As explained in the introduction, this implies that the following  inequality  
\begin{equation*}
\Vert u \Vert_{L^2(\Gamma)} \leq  C(\Omega) \Vert u \Vert_{H^{1/2}(\Omega)}
\end{equation*}
cannot be satisfied in general for harmonic functions $H^{1/2}(\Omega)$ when the domain is only Lipschitz and  the same applies to inequalities  \eqref{inegaltraceL2Gamma1} and \eqref{ineg1}.

\section{Traces in the limit cases $H^{1/2}(\Omega)$ and $H^{3/2}(\Omega)$} \label{traces}
The questions of traces of functions belonging to Sobolev spaces are fundamental in the study of boundary value problems. Classically, for bounded Lipschitz domains of $\mathbb{R}^N$, we know that for any $1/2< s < 3/2$ the linear mapping 
\begin{equation}\label{trlim}
\gamma_0:  u \longrightarrow u_{\vert \Gamma}\\
\end{equation}
belongs to $\mathscr{L}(H^s(\Omega); H^{s-1/2}(\Gamma))$, but this continuity property is wrong for $s = 1/2$ or $s = 3/2$. Moreover if $u\in H^s(\Omega)$  for $s > 3/2$ (resp. $u\in H^{3/2}(\Omega)$), then $u_{\vert \Gamma}\in H^1(\Gamma)$ (resp. $u_{\vert \Gamma}\in H^{1-\varepsilon}(\Gamma)$ for any $0 < \varepsilon \leq 2$). We will investigate in this section the crucial limit cases $s = 1/2$,  $s = 3/2$. And if $u\in H^{1/2}(\Omega)$ (resp. $u\in H^{3/2}(\Omega))$, what additional condition can be added to $u$ so that its trace belongs to $L^2(\Gamma)$ (resp. to $H^1(\Gamma)$). Recall that when $\Omega$ is of class $\mathscr{C}^{k-1, 1}$, with positive integer $k$, then the above mapping \eqref{trlim} is continuous for any  $1/2< s < 1/2 + k$ (see \cite{Wend}). In particular if $\Omega$ is of class $\mathscr{C}^{1, 1}$ and $u\in H^{3/2}(\Omega)$, then its trace belongs to $H^1(\Gamma)$ (see also \cite{McL}, Theorem 3.7). \medskip

Now, concerning the normal derivative $\partial_{\textit{\textbf n}}u$ of $u$, we recall that if $\Omega$ is a bounded Lipschitz domain and $u\in H^{3/2}(\Omega)$, the normal derivative is in general not defined (even if $\Omega$ is more regular). However, when $u\in H^{s}(\Omega)$, with $s >  3/2$, then $\partial_{\textit{\textbf n}}u \in L^2(\Gamma)$ only. In the case where $\Omega$ is $\mathscr{C}^{1, 1}$  and $u \in H^s(\Omega)$, with $3/2 < s < 5/2$, then $\partial_{\textit{\textbf n}}u\in H^{s-3/2}(\Gamma)$.\medskip

Recall also that for any  $v\in H^{s}(\Omega)$, with $0 < s < 1$, we have $\nabla v \in \textit{\textbf H}\, ^{s-1}(\Omega) $ if $s \neq 1/2$ and $\nabla v \in [\textit{\textbf H}\, ^{1/2}_{00}(\Omega)]' $ if $s = 1/2$, this last dual space being bigger that the dual space $[\textit{\textbf H}\, ^{1/2}(\Omega)]' $ (see Theorem 1.4.4.6 and Proposition 1.4.4.8 in \cite{Gri}). Note that $\mathscr{D}(\Omega)$ is dense  in $H^{1/2}(\Omega)$ and in $H^{1/2}_{00}(\Omega)$ respectively. So the dual spaces $[H^{1/2}(\Omega)]'$ and  $[H^{1/2}_{00}(\Omega)]'$ are both subspaces of  $\mathscr{D}'(\Omega)$. Moreover, since  $H^{1/2}_{00}(\Omega)\hookrightarrow H^{1/2}(\Omega)$, with density, we have the continuous embedding $  [H^{1/2}(\Omega)]' \hookrightarrow [H^{1/2}_{00}(\Omega)]' $.

\begin{lemma}\label{DensityDOmegabarH1/2GradH1/2'}
The space $\mathscr{D}(\overline\Omega)$ is dense in the following space:
\begin{equation*}\label{defEnablaOmega} 
 E(\nabla;\, \Omega)\ =\ \left\{\,v\in H^{1/2}(\Omega);\ \nabla v\in  [\textit{\textbf H}^{\, 1/2}(\Omega)]'\,\right\}.
\end{equation*}
\end{lemma}

\begin{proof} The proof of the density of  $\mathscr{D}(\overline\Omega)$ in $E(\nabla;\, \Omega)$ is similar to that of $\mathscr{D}(\overline\Omega)$ in $H^{1}(\Omega)$, but little bit more complicated. It suffices to consider the case where $\Omega = \mathbb{R}^N_+$ is the half space. 
\medskip

\noindent{\bf Step 1}. {\it We will prove that the functions of $E(\nabla;\, \Omega)$ with compact support is dense in $E(\nabla;\, \Omega)$.} Let $\psi \in \mathscr{D}(\mathbb{R}^N)$, with 
\begin{equation*} 
\psi(x) = \begin{cases} 1 \quad \; \mathrm{if}\;\;\vert x \vert  \leq 5/4\\
0 \quad \; \mathrm{if}\;\;\vert x \vert  \geq 7/4
\end{cases}
\end{equation*}
and define 
$$
\mathrm{for\, \, any} \, k\in \mathbb{N}^*,\quad \psi_k (x) = \psi(x/k).
$$
For $v\in  H^{1/2}(\mathbb{R}^N_+)$, setting $v_k = \psi_k v$, we can prove by some direct calculations the following estimate:
\begin{equation}\label{estimH1/2H1/2primedemiespace}
\Vert v_k - v \Vert_{H^{1/2}(\mathbb{R}^N_+)}\leq C\big( \Vert v\Vert_{H^{1/2}(\mathbb{R}^N_+\cap B_k^c)}  + \frac{1}{\sqrt k}  \Vert v\Vert_{L^{2}(\mathbb{R}^N_+)} \big),
\end{equation}
where $B_k^c$ is the the complementary of $B_k$ in the whole space. Besides for any $\varphi \in \mathscr{D}(\mathbb{R}^N_+)$ and $j = 1, \ldots, N$, we have
\begin{equation*}
\langle \frac{\partial}{\partial x_{j}}( v_k - v),\, \varphi \rangle = \langle \frac{\partial v}{\partial x_{j}}, \, \psi_k \varphi - \varphi \rangle + \int_{\mathbb{R}^N_+}v \varphi \frac{\partial \psi_k}{\partial x_{j}}
\end{equation*}
and then by using \eqref{estimH1/2H1/2primedemiespace}, we get
\begin{equation*}\label{inegbracketH12H1/2prime}
\begin{array}{rl}
 \vert \langle \frac{\partial}{\partial x_{j}}( v_k - v),\, \varphi \rangle \vert & \leq C  \Vert \frac{\partial v}{\partial x_{j}}\Vert_{ [\textit{\textbf H}^{\, 1/2}(\mathbb{R}^N_+)]'} \big( \Vert \varphi_{H^{1/2}(\mathbb{R}^N_+\cap B_k^c)}  + \frac{1}{\sqrt k}  \Vert \varphi\Vert_{L^{2}(\mathbb{R}^N_+)} \big) \\
&  + \; \frac{C}{k} \Vert v \Vert_{L^{2}(\mathbb{R}^N_+)}\Vert \varphi \Vert_{L^{2}(\mathbb{R}^N_+)},
\end{array}
\end{equation*}
Hence, 
$$
\frac{\partial v_k}{\partial x_{j}}  \rightharpoonup \frac{\partial v}{\partial x_{j}} \quad \mathrm{in}\; [H^{\, 1/2}(\mathbb{R}^N_+)]'.
$$
Our goal is to prove the strong convergence. For that, we observe that $\frac{\partial v_k}{\partial x_{j}} = \psi_k \frac{\partial v}{\partial x_{j}} + v\frac{\partial \psi_{k}}{\partial x_{j}}$ and $v\frac{\partial \psi_{k}}{\partial x_{j}} \rightarrow 0 $ in $L^2(\mathbb{R}^N_+)$ and then in $ [H^{\, 1/2}(\mathbb{R}^N_+)]'$. In addition, since for any $\varphi \in H^{1/2}(\mathbb{R}^N_+)$
$$
 \vert \langle \psi_k \frac{\partial v}{\partial x_{j}} ,\, \varphi \rangle\vert = \vert \langle \frac{\partial v}{\partial x_{j}} ,\, \psi_k\varphi \rangle \vert  \leq \Vert  \frac{\partial v}{\partial x_{j}} \Vert_{[H^{\, 1/2}(\mathbb{R}^N_+)]'} \Vert \psi_k \varphi\Vert_{H^{\, 1/2}(\mathbb{R}^N_+)}
$$
we have the following estimate
$$
\Vert \psi_k \frac{\partial v}{\partial x_{j}} \Vert_{[H^{\, 1/2}(\mathbb{R}^N_+)]'} \leq \Vert  \frac{\partial v}{\partial x_{j}} \Vert_{[H^{\, 1/2}(\mathbb{R}^N_+)]'} \sup_{\varphi\in H^{1/2}(\mathbb{R}^N_+), \varphi \not= 0}\frac{\Vert \psi_k \varphi\Vert_{H^{\, 1/2}(\mathbb{R}^N_+)}}{\Vert  \varphi\Vert_{H^{\, 1/2}(\mathbb{R}^N_+)}}.
$$
As 
$$
\limsup_{k\rightarrow \infty}\Vert \psi_k \frac{\partial v}{\partial x_{j}} \Vert_{[H^{\, 1/2}(\mathbb{R}^N_+)]'} \leq \Vert \frac{\partial v}{\partial x_{j}} \Vert_{[H^{\, 1/2}(\mathbb{R}^N_+)]'},
$$
we have also the same inequality for the norm of $\frac{\partial v_k}{\partial x_{j}}$ and then we deduce the desired strong convergence.\medskip

\noindent{\bf Step 2}.  {\it Extension to} $ \mathbb{R}^N$. It follows from Step 1 that we can suppose, without loss of generality, that $v\in H^{1/2}(\mathbb{R}^N_+)$ with compact support.

For $h > 0$  we set $\tau_h v(\textit{\textbf x})  = v_h (\textit{\textbf x}) = v(\textit{\textbf x}', x_N + h)$ and we introduce the following function
\begin{equation*}
\alpha_h(\textit{\textbf x}) = \begin{cases} 1 \quad \; \mathrm{if}\;\; x_N > 0\\
0 \quad \; \mathrm{if}\;\; x_N < - h
\end{cases}
\end{equation*}
with $\alpha_h \in \mathscr{C}^1(\mathbb{R}^N)$. We set $w_h = \alpha_h \tau_h Pu$, where $P :  H^{1/2}(\mathbb{R}^N_+) \rightarrow H^{1/2}(\mathbb{R}^N)$ is a bounded linear extension operator. Clearly,  if $v \in H^{1/2}(\mathbb{R}^N_+)$, using Lebesgue's dominated convergence theorem, then we have $  v_h \rightarrow v$ in $H^{1/2}(\mathbb{R}^N_+)$ and $  w_h{_{\vert \mathbb{R}^N_+}} \rightarrow v$ in $H^{1/2}(\mathbb{R}^N_+)$  as  $h  \rightarrow 0$. Moreover, for any $\varphi \in \mathscr{D}(\mathbb{R}^N_+)$ and $j = 1, \ldots, N$,
\begin{equation*}
\vert \langle \frac{\partial w_h}{\partial x_{j}} ,\, \varphi \rangle \vert = \vert \langle \frac{\partial u}{\partial x_{j}},\,   \tau_{-h} \varphi\rangle \vert  \leq \Vert \frac{\partial u}{\partial x_{j}}\Vert_{[H^{\, 1/2}(\mathbb{R}^N_+)]'} \Vert \varphi \Vert_{H^{\, 1/2}(\mathbb{R}^N_+)}.
\end{equation*}
Using the density of $\mathscr{D}(\mathbb{R}^N_+)$ in $H^{\, 1/2}(\mathbb{R}^N_+)]$, we deduce that 
\begin{equation*}
\frac{\partial w_h}{\partial x_{j}} \in [H^{\, 1/2}(\mathbb{R}^N_+)]'\quad \mathrm{and}\quad \Vert \frac{\partial w_h}{\partial x_{j}}\Vert_{[H^{\, 1/2}(\mathbb{R}^N_+)]'} \leq  \Vert \frac{\partial u}{\partial x_{j}}\Vert_{[H^{1/2}(\mathbb{R}^N_+)]'}
\end{equation*}
(where the last inequality can be obtained by interpolation between $L^2(\mathbb{R}^N_+)$ and $H^{-1}(\mathbb{R}^N_+)$).
Besides, for any $\varphi \in \mathscr{D}(\mathbb{R}^N_+)$ and $j = 1, \ldots, N$, we have even if it means extending $\varphi$ by zero outside the half-space,
\begin{equation*}
\vert \langle \frac{\partial w_h}{\partial x_{j}} -  \frac{\partial u}{\partial x_{j}} ,\, \varphi \rangle \vert = \vert\langle \frac{\partial u}{\partial x_{j}},\,   \tau_{-h}\varphi - \varphi \rangle \vert  \leq  \Vert \frac{\partial u}{\partial x_{j}}\Vert_{[H^{1/2}(\mathbb{R}^N_+)]'} \Vert   \tau_{-h}\varphi - \varphi \Vert_{H^{\, 1/2}(\mathbb{R}^N_+)} 
\end{equation*}
where the last norm above tends to $0$  when $h\rightarrow 0$. That gives the strong convergence  
$$
 \frac{\partial w_h}{\partial x_{j}} \rightarrow \frac{\partial u}{\partial x_{j}} \qquad \mathrm{ in }\; [H^{1/2}(\mathbb{R}^N_+)]'.
$$
\medskip
\noindent{\bf Step 3}.  {\it Regularization}. To finish, we will approximate $w_h$, with $h$ fixed, by the functions $\varphi_k = w_h\star \varrho_k$, where we use the sequence of mollifiers $(\varrho_k)_k$. It is easy to verify that 
$$
\varphi_k \rightarrow w_h \qquad \mathrm{ in }\; H^{1/2}(\mathbb{R}^N)  \qquad \mathrm{ and } \qquad  \frac{\partial \varphi_k}{\partial x_{j}} \rightarrow  \frac{\partial w_h }{\partial x_{j}} \qquad \mathrm{ in }\; [H^{1/2}(\mathbb{R}^N )]'
$$ 
as  $k  \rightarrow \infty$. 
\end{proof}

We saw in the previous section that functions $H^{1/2}(\Omega)$ generally do not have an $L^2(\Gamma)$ trace, even when they are harmonic. We will see a little later that this is the case when the domain $\Omega$ is more regular. The following lemma gives a sufficient condition for $H^{1/2}(\Omega)$ functions to actually have an $L^2(\Gamma)$ trace and allows us to give a new characterization of the space $H^{1/2}_{00}(\Omega)$.

\begin{theorem}[{\bf Trace operator on $E(\nabla;\, \Omega)$}]\label{TracesH1demigradH1demiprime} i) The linear mapping $\gamma_0: u \mapsto u_{\vert \Gamma}$ defined on $\mathscr{D}(\overline{\Omega})$ can be extended by continuity to a linear and continuous mapping, still denoted $\gamma_0$, from $E(\nabla;\, \Omega)$ into $L^2(\Gamma)$. \\
ii) The kernel of  $\gamma_0: u \mapsto u_{\vert \Gamma}$ from $E(\nabla;\, \Omega)$ into $L^2(\Gamma)$ is equal to $H^{1/2}_{00}(\Omega)$.
\end{theorem}

\begin{proof} i) For any $v\in \mathscr{D}(\overline{\Omega})$ 
\begin{equation*}
 \int_\Gamma \textbf{\textit h}\cdot\textbf{\textit n}\vert v \vert^2 = 2 \int_\Omega v\nabla v \cdot \textbf{\textit h} +  \int_\Omega \vert v \vert^2 \mathrm{div}\,  \textbf{\textit h},
\end{equation*}
where $\textbf{\textit h}\in \mathscr{C}^\infty(\overline{\Omega})$ is such that   $\textbf{\textit h} \cdot\textbf{\textit n}\geq \alpha > 0$ a.e  on $\Gamma$ (see Lemma 1.5.1.9 in \cite{Gri}).
Consequently, we have the following estimate:
\begin{equation*}
 \Vert v \Vert^2_{L^2(\Gamma)}\leq C (\Vert \nabla v \Vert_{ [\textbf{\textit H}\, ^{1/2}(\Omega)]'} \Vert v \Vert_{H^{1/2}(\Omega)} + \Vert v \Vert^2_{L^2(\Omega)} ),
\end{equation*}
which means that
\begin{equation*}
 \Vert v \Vert_{L^2(\Gamma)}\leq C \Vert  v \Vert_{E(\nabla;\, \Omega)},
\end{equation*}
where the space $E(\nabla;\, \Omega)$ is equipped with the graph norm.\medskip

The required property is finally a consequence of the density of $\mathscr{D}(\overline{\Omega})$ in $E(\nabla;\, \Omega)$. \smallskip

\noindent ii) Observe that $H^{1/2}_{00}(\Omega)$ is included in $E(\nabla;\, \Omega)$. So by using the density of $\mathscr{D}(\Omega)$ in $H^{1/2}_{00}(\Omega)$,  we have the following inclusion: $H^{1/2}_{00}(\Omega) \subset Ker \, \gamma_0$.  \medskip

Conversely, let $u \in E(\nabla;\, \Omega)$ with $u = 0$ on $\Gamma$ and $\widetilde{u}$ the extension by $0$ of $u$ outside of $\Omega$. Then for any $j = 1, \ldots, N$ any $\varphi \in \mathscr{D}(\mathbb{R}^N)$ we have
$$ 
\langle  \frac{\partial \widetilde{u}}{\partial x_{j}}, \varphi\rangle =  - \int_{\mathbb{R}^N}  \widetilde{u}\frac{\partial \varphi}{\partial x_{j}}=  - \int_{\Omega}  u\frac{\partial \varphi}{\partial x_{j}}.
$$
Now, using the density of  $\mathscr{D}(\overline{\Omega})$ in $E(\nabla;\, \Omega)$, there exists a sequence $u_k$ in $\mathscr{D}(\overline{\Omega})$ such $u_k \rightarrow u$ in  $E(\nabla;\, \Omega)$. Hence
$$
\int_{\Omega}  u\frac{\partial \varphi}{\partial x_{j}} = \lim_{k\rightarrow \infty}\int_\Omega u_k \frac{\partial \varphi}{\partial x_{j}} = \lim_{k\rightarrow \infty}(- \int_\Omega \varphi \frac{\partial u_k} {\partial x_{j}} + \int_\Gamma u_k \varphi).
$$
Thanks to Point i) we know that $u_{k_{\vert \Gamma}} \rightarrow 0$ in $L^2(\Gamma)$. As $ \frac{\partial u_k} {\partial x_{j}} \rightarrow  \frac{\partial u} {\partial x_{j}}$ in $[H^{1/2}(\Omega)]'$, we deduce from above that  
$$
- \int_{\Omega}  u\frac{\partial \varphi}{\partial x_{j}} = \langle  \frac{\partial u}{\partial x_{j}}, \varphi\rangle_{ [H^{\, 1/2}(\Omega)]'\times H^{\, 1/2}(\Omega)}.
$$
So we have
$$\vert\langle  \frac{\partial \widetilde{u}}{\partial x_{j}}, \varphi\rangle \vert \leq \Vert \frac{\partial u}{\partial x_{j}} \Vert_{[H^{\, 1/2}(\Omega)]'}\Vert \varphi \Vert_{H^{\, 1/2}(\Omega)} \leq  \Vert \frac{\partial u}{\partial x_{j}} \Vert_{[H^{\, 1/2}(\Omega)]'}\Vert \varphi \Vert_{H^{\, 1/2}(\mathbb{R}^N)},
$$
 which means that
$$
 \nabla  \widetilde{u}\in  \textit{\textbf H}^{\, -1/2}(\mathbb{R}^N).
 $$
Since $\widetilde{u} - \Delta \widetilde{u} \in  H^{\, -3/2}(\mathbb{R}^N)$, then $ \widetilde{u}\in  H^{\, 1/2}(\mathbb{R}^N)$ and therefore $u\in H^{1/2}_{00}(\Omega)$. 
\end{proof}

\begin{remark}\label{Surj}\upshape What about the characterization of the range of $E(\nabla;\, \Omega)$ by the linear mapping $\gamma_0: u \mapsto u_{\vert \Gamma}$? Is this range equal or strictly included in $L^2(\Gamma)$? We will answer this question a little later on, see Proposition \ref{protrace}.
\end{remark}

\begin{corollary}\label{TracesbH1demigradH1demiprime} i) The linear mapping $\gamma: u \mapsto (u_{\vert \Gamma},  \frac{\partial u}{\partial\textit{\textbf n}})$ is  continuous from $E(\nabla^2;\, \Omega)$ into $H^1(\Gamma)\times L^2(\Gamma)$, where
\begin{equation*}\label{defEnabla2Omega}
 E(\nabla^2;\, \Omega)\ =\ \left\{\,v\in H^{3/2}(\Omega);\ \nabla ^2v\in  [\textit{\textbf H}^{\, 1/2}(\Omega)]'\,\right\}.
\end{equation*}
ii) The kernel of  $\gamma$ from $E(\nabla^2;\, \Omega)$ into  $H^1(\Gamma)\times L^2(\Gamma)$ is equal to $H^{3/2}_{00}(\Omega)$.
\end{corollary}

We will now focus on questions related to the normal derivative of $H^1(\Omega)$ functions with sufficiently regular Laplacian.  Recall that if $v \in H^1(\Omega)$ with $\Delta v \in L^2(\Omega)$ (or even $\Delta v \in [H^{1/2}(\Omega)]'$), then the normal derivative  $\partial_\textit{\textbf n} v = \nabla v\cdot\textit{\textbf n}  \in H^{-1/2}(\Gamma)$ and we have the following Green formula:
\begin{equation*}
	\forall\varphi\in H^{1}(\Omega),\quad  \int_{\Omega}(\nabla v\cdot \nabla \varphi + \varphi \Delta v) = \langle \partial_\textit{\textbf n} v, \varphi \rangle_{H^{-1/2}(\Gamma) \times H^{1/2}(\Gamma)}.
\end{equation*}
To prove the existence of this normal derivative and this Green formula, we can use the surjectivity of  the trace operator from $H^1(\Omega)$ into $H^{1/2}(\Gamma)$,  the same Green formula for smooth functions $v$ in $\mathscr{D}(\overline{\Omega})$ and finally the density of this last space into $E(\Delta, \Omega) = \{v \in H^1(\Omega); \;  \Delta v \in L^2(\Omega)\}$. Or more simply, as in \cite{Fabes3},  by defining the normal derivative as the  continuous linear form on $H^{1/2}(\Gamma)$ 
\begin{equation*}
T: g \mapsto \int_{\Omega}(\nabla v\cdot \nabla \varphi + \varphi \Delta v), 
\end{equation*}
 where $\varphi \in H^1(\Omega)$ is any extension of $g$ satisfying $\Vert \varphi \Vert_{H^1(\Omega)} \leq C(\Omega) \Vert g \Vert_{H^{1/2}(\Gamma)}$.

\begin{remark} \upshape \label{remDN} What happens now for the regularity of $ \partial_\textit{\textbf n} v$ if in addition  $v\in H^{3/2}(\Omega)$ and $\Delta v = 0$ in $\Omega$?  Clearly, we know that $ \partial_\textit{\textbf n} v \in H^{-s}(\Gamma)$ for any $0 < s \leq 1/2$. So the precise question is now: can we take $s = 0$ and then to get $ \partial_\textit{\textbf n} v \in L^2(\Gamma)$?  For that  we must show that the above linear form $T$ can be extended (and in this case the extension will be unique since  $H^{1/2}(\Gamma)$ is dense in $L^2(\Gamma)$) to a continuous linear form on $L^2(\Gamma)$ still denoted by $T$:
\begin{equation*}
T:  g \mapsto \langle \nabla v,  \nabla u_g\rangle_{H^{1/2}(\Omega)\times [H^{1/2}(\Omega)]'} , 
\end{equation*}
where $u_g \in H^{1/2}(\Omega)$ is the unique harmonique function satisfying $u_g = g$ on $\Gamma$. Recalling that the linear operator $S : g \mapsto u_g$ is continuous from $L^2(\Gamma)$ into $H^{1/2}(\Omega)$, the answer to the previous question will be positive if and only if the following linear mapping
\begin{equation}\label{nablaS}
\nabla_\circ S : L^2(\Gamma) \longrightarrow [H^{1/2}(\Omega)]' 
\end{equation}
is  continuous. We will answer this question in the next theorem.\medskip
\end{remark}

\begin{theorem} [{\bf Traces of harmonic functions  $H^{1/2}$}]\label{protrace}
Assume that $\Omega$ is of class $\mathscr{C}^{1, 1}$. We have the following properties:\\
i) Let $u \in H^{1/2}(\Omega)$. Then we have the following implication:
\begin{equation}\label{impharC11}
u \quad \mathrm{harmonic}\quad   \Longrightarrow \quad u_{\vert\Gamma} \in L^2(\Gamma) \quad \mathrm{and} \quad \nabla u \in  [\textit{\textbf  H}^{\, 1/2}(\Omega)]'
\end{equation} 
In particular, the trace operator $\gamma :  H^{1/2}(\Omega)\cap \mathscr{H} \rightarrow L^2(\Gamma)$ is an isomorphism, where $\mathscr{H}$ is the space of harmonic functions in $\Omega$. Moreover, the range $\gamma(E(\nabla, \Omega))$ is equal to $ L^2(\Gamma)$. \\
ii) We have the algebraical identity $H^{1/2}(\Omega)\cap \mathscr{H}  = E(\nabla, (\Omega))\cap \mathscr{H}$ and topological:
$$
\Vert \nabla v \Vert_{[\textit{\textbf H}^{1/2}(\Omega)]'} \approx \Vert \nabla v \Vert_{[\textit{\textbf H}^{1/2}_{00}(\Omega)]'} \quad \mathrm{for}\; v \in H^{1/2}(\Omega)\cap \mathscr{H}.
$$
iii) The operator \eqref{nablaS} is  continuous.
\end{theorem}

\begin{proof} {\bf i) Step 1.} We know that $S \in \mathscr{L}(H^{1/2}(\Gamma); H^1(\Omega)\cap \mathscr{H})$. Since $\Omega$ is of class $\mathscr{C}^{1, 1}$ we have also $S \in \mathscr{L}(H^{-1/2}(\Gamma); L^2(\Omega)\cap \mathscr{H})$ (see Theorem 7 in  \cite{ARB}). As
$$
[L^2(\Omega)\cap \mathscr{H}, H^1(\Omega)\cap \mathscr{H}]_{1/2} = H^{1/2}(\Omega)\cap \mathscr{H},
$$
(see \cite{J-K} page 183) we deduce  by interpolation that $S \in \mathscr{L}(L^{2}(\Gamma); H^{1/2}(\Omega)\cap \mathscr{H})$. Note that the first and  third operators above, unlike the second, are continuous even if $\Omega$ is only Lipschitz.  On the other hand, the trace operator $\gamma$ satisfies:
$$
\forall v\in H^1(\Omega)\cap \mathscr{H}, \quad \Vert \gamma v \Vert_{H^{1/2}(\Gamma)} \leq C(\Omega) \Vert v \Vert_{H^1(\Omega)}\quad 
$$
and
$$
\forall v\in L^2(\Omega)\cap \mathscr{H}, \quad \Vert \gamma v \Vert_{H^{-1/2}(\Gamma)} \leq C(\Omega) \Vert v \Vert_{L^2(\Omega)},
$$
(see Lemma 2 in  \cite{ARB} for this last inequality which holds since $  \Omega$ is $ \mathscr{C}^{1, 1}$).
By interpolation again, we deduce that
\begin{equation*}
\forall v\in H^{1/2}(\Omega)\cap \mathscr{H}, \quad \Vert \gamma v \Vert_{L^{2}(\Gamma)} \leq C(\Omega) \Vert v \Vert_{H^{1/2}(\Omega)}.
\end{equation*} 
That means that the trace operator $\gamma :  H^{1/2}(\Omega)\cap \mathscr{H} \rightarrow L^2(\Gamma)$ is continuous and is bijective.   Moreover $S = \gamma^{-1}$ and 
$$
 \Vert v \Vert_{H^{1/2}(\Omega)} \approx \Vert \gamma v \ \Vert_{L^{2}(\Gamma)} \quad \mathrm{for}\quad  v\in H^{1/2}(\Omega)\cap \mathscr{H}.
 $$
 This proves the first part of implication \eqref{impharC11}. \medskip
  
{\bf Step 2.}  Let us now prove the second part of implication \eqref{impharC11}. Let  $u \in H^{1/2}(\Omega)$ a harmonic function. Since $\Omega$ is of class $\mathscr{C}^{1, 1}$,  for any $\textit{\textbf F} \in \mathscr{D}(\Omega)^N$ there exists a unique 
 $\chi \in H^2(\Omega)\cap L^2_0(\Omega)$ such that 
\begin{equation}\label{estimNeu}
\Delta \chi = \mathrm{div}\, \textit{\textbf F} \quad \mathrm{in}\; \Omega \quad \mathrm{and}\quad  \partial_\textit{\textbf n}\chi = 0\quad \mathrm{on}\; \Gamma, \quad \mathrm{with}\quad \Vert \chi \Vert_{H^{3/2}(\Omega)} \leq C \Vert \textit{\textbf F}\,  \Vert_{\textit{\textbf  H}^{\, 1/2}(\Omega)}. 
\end{equation}
Using the harmonicity of $u$, we have
$$
\langle \nabla u, \textit{\textbf F}\, \rangle_{[\mathscr{D}'(\Omega)]^N \times \mathscr{D}(\Omega)^N} = - \int_\Omega u\,  \mathrm{div}\, \textit{\textbf F} = - \int_\Omega u\,  \Delta \chi.
$$
Since the regularity of the domain $\Omega$, we know that for any $ v\in L^2(\Omega)$ with $ \Delta v \in L^2(\Omega)$ and $\varphi \in H^2(\Omega)$, 
\begin{equation*}
 \int_\Omega v\Delta \varphi -  \int_\Omega \varphi \Delta v = \langle v, \partial_\textit{\textbf n} \varphi\rangle_{H^{-1/2}(\Gamma)\times H^{1/2}(\Gamma)} -  \langle \partial_\textit{\textbf n} v , \varphi \rangle_{H^{-3/2}(\Gamma)\times H^{3/2}(\Gamma)}.
\end{equation*}
Recall also that the Steklov operator (also called the Dirichlet-to-Neumann operator) denoted by $S_P$ satisfies the estimate (which holds even if $\Omega$ is only Lipschitz):
$$
\Vert S_P u \Vert_{H^{-1}(\Gamma)} \leq C(\Omega) \Vert  u \Vert_{L^{2}(\Gamma)}.
$$

So we deduce that
$$
\langle \nabla u, \textit{\textbf F}\, \rangle_{[\mathscr{D}'(\Omega)]^N \times \mathscr{D}(\Omega)^N}  = - \langle \partial_\textit{\textbf n} u, \chi\rangle_{H^{-1}(\Gamma)\times H^1(\Gamma)}
 $$
and then
$$
\vert \langle \nabla u, \textit{\textbf F}\, \rangle_{[\mathscr{D}'(\Omega)]^N \times \mathscr{D}(\Omega)^N} \vert \leq C(\Omega) \Vert u \Vert_{L^{2}(\Gamma)} \Vert \chi \Vert_{H^{1}(\Gamma)}.  
$$
As $\Omega$ is $\mathscr{C}^{1, 1}$, we know that
$$
\Vert \chi \Vert_{H^{1}(\Gamma)} \leq C(\Omega)\Vert \chi \Vert_{H^{3/2}(\Omega)}.
$$
Note that this last inequality does not occur when $\Omega$ is only Lipchitz. From the estimate in \eqref{estimNeu}, we finally deduce that
$$
\vert \langle \nabla u, \textit{\textbf F}\, \rangle_{[\mathscr{D}'(\Omega)]^N \times \mathscr{D}(\Omega)^N} \vert \leq C(\Omega)\Vert u \Vert_{L^2(\Gamma)} \Vert \textit{\textbf F}\,  \Vert_{\textit{\textbf  H}^{\, 1/2}(\Omega)}
$$
and then the estimate
$$
\Vert \nabla u \Vert_{[\textit{\textbf  H}^{\, 1/2}(\Omega)]'} \leq C(\Omega)\Vert u \Vert_{L^2(\Gamma)} \leq C(\Omega)\Vert u \Vert_{H^{1/2}(\Omega)} 
$$
by using the density of $\mathscr{D}(\Omega)^N$ in $\textit{\textbf  H}^{\, 1/2}(\Omega)$. This proves the second part of implication \eqref{impharC11}. \medskip

The isomorphism $\gamma :  H^{1/2}(\Omega)\cap \mathscr{H} \rightarrow L^2(\Gamma)$ is a simple consequence of the fact that for every $g\in L^2(\Gamma)$, there exists a unique harmonic function $w$ in $\Omega$ satisfying $w = g$ on $\Gamma$  (see Theorem 5.3 in \cite{J-K} or Theorem 8.4 in \cite{AM}). And we clearly have the identity $\gamma(E(\nabla, \Omega)) = L^2(\Gamma)$.\medskip

{\bf ii)} It is an immediate consequence of \eqref{impharC11}.\medskip

{\bf iii)}  The last property concerning the operator \eqref{nablaS} is an immediate consequence of Point ii) and the continuity of $S$.
\end{proof}

 \medskip

Recall the following Ne$\mathrm{\check{c}}$as property (see \cite{Necas}, Chapter 5): if $u\in H^1(\Omega)$ with $\Delta u  \in\ L^2(\Omega)$,  then we have the following equivalences 
\begin{equation}\label{NecProb}
u\in H^1(\Gamma) \Longleftrightarrow \frac{\partial u}{\partial \textit{\textbf n}}\in L^2(\Gamma) \Longleftrightarrow  \nabla u \in L^2(\Gamma).
\end{equation}
Moreover, in this case, we have the following estimates
\begin{equation*}
\Vert\frac{\partial u}{\partial \textit{\textbf n}}\Vert_{ \textit{\textbf L}^2(\Gamma)}\ \leq\ C(\Omega)\Big(\inf_{k\in  \mathbb{R}}\left|\left |u + k \right |\right |_{H^1(\Gamma)} + \Vert \Delta u\Vert_{L^2(\Omega)}\Big)
\end{equation*}
and
\begin{equation*}
 \inf_{k\in  \mathbb{R}}\Vert u + k\Vert_{H^1(\Gamma)}  \ \leq\ C(\Omega)\Big( \Vert\frac{\partial u}{\partial \textit{\textbf n}}\Vert_{ \textit{\textbf L}^2(\Gamma)} +  \Vert \Delta u\Vert_{L^2(\Omega)}\Big).
\end{equation*}
Observe that the above equivalences \eqref{NecProb} are valid if we replace the condition  $\Delta u  \in\ L^2(\Omega)$ by the weaker condition  $\Delta u  \in [H^{1/2}(\Omega)]'$ (see Corollary 8.9 in \cite{AM}).

\begin{theorem}  [{\bf Traces of harmonic functions  $H^{3/2}$}]\label{protrace32}
Assume that $\Omega$ is of class $\mathscr{C}^{1, 1}$. We have the following properties: let $u \in H^{3/2}(\Omega)$, then we have the following implication:
\begin{equation}\label{impharC1132}
u \quad \mathrm{harmonic}\quad   \Longrightarrow \quad \partial_\textit{\textbf n} u \in L^2(\Gamma) \quad \mathrm{and} \quad \nabla^2 u \in  [\textit{\textbf  H}^{\, 1/2}(\Omega)]'
\end{equation} 
In particular, the trace operator $\gamma_{\textit{\textbf n}} :  v \rightarrow \partial_\textit{\textbf n} v$ is an isomorphism from $ H^{3/2}(\Omega)\cap L^2_0(\Omega) \cap \mathscr{H} $ onto $ L^2_0(\Gamma)$. 
\end{theorem}

\begin{proof} We start by observe that $u \in H^1(\Gamma) $ since $\Omega$ is of class $\mathscr{C}^{1, 1}$. So the first part of implication \eqref{impharC1132} is an immediate consequence of \eqref{NecProb}. Another way to obtain this result is to proceed as in the proof of Theorem \ref{protrace}.\\

 Let us now prove the second part of implication \eqref{impharC1132}. Let  $u \in H^{3/2}(\Omega)$ a harmonic function. Since $\Omega$ is of class $\mathscr{C}^{1, 1}$,  for any $f \in \mathscr{D}(\Omega)$ and any $j = 1, \ldots, N$, there exists a unique 
 $\chi_j \in H^2(\Omega)\cap L^2_0(\Omega)$ such that 
\begin{equation}\label{estimNeuj}\
\Delta \chi_j = \frac{\partial f}{\partial x_{j}} \quad \mathrm{in}\; \Omega \quad \mathrm{and}\quad  \partial_\textit{\textbf n}\chi_j = 0\quad \mathrm{on}\; \Gamma, \quad \mathrm{with}\quad \Vert \chi_j \Vert_{H^{3/2}(\Omega)} \leq C_j \Vert f  \Vert_{H^{1/2}(\Omega)}. 
\end{equation}
Using the harmonicity of $u$, we have for any $i = 1, \ldots, N$, 
$$
\langle \frac{\partial^2  u}{\partial x_{i}\partial x_{j}}, f\, \rangle_{\mathscr{D}'(\Omega) \times \mathscr{D}(\Omega)} = - \int_\Omega \frac{\partial u}{\partial x_{i}} \, \frac{\partial f}{\partial x_{j}} = - \int_\Omega \frac{\partial u}{\partial x_{i}}  \Delta \chi_j .
$$
Since the regularity of the domain $\Omega$, we know that for any $ v\in L^2(\Omega)$ with $ \Delta v \in L^2(\Omega)$ and $\varphi \in H^2(\Omega)$, 
\begin{equation*}
 \int_\Omega v\Delta \varphi -  \int_\Omega \varphi \Delta v = \langle v, \partial_\textit{\textbf n} \varphi\rangle_{H^{-1/2}(\Gamma)\times H^{1/2}(\Gamma)} -  \langle \partial_\textit{\textbf n} v , \varphi \rangle_{H^{-3/2}(\Gamma)\times H^{3/2}(\Gamma)}.
\end{equation*}
So we deduce from this Green formula that
$$
\langle \frac{\partial^2  u}{\partial x_{i}\partial x_{j}}, f \rangle_{\mathscr{D}'(\Omega) \times \mathscr{D}(\Omega)}   = - \langle \partial_\textit{\textbf n} (\frac{\partial u}{\partial x_{i}}), \chi_j\rangle_{H^{-1}(\Gamma)\times H^1(\Gamma)}.
 $$
Note that $\frac{\partial u}{\partial x_{i}} \in H^{1/2}(\Omega)$ and is harmonic. Hence $\frac{\partial u}{\partial x_{i}}_{\vert\Gamma} \in L^2(\Gamma)$ and $\partial_\textit{\textbf n} (\frac{\partial u}{\partial x_{i}}) \in H^{-1}(\Gamma)$. As in the proof of Theorem \ref{protrace}, we deduce that
$$
\vert \langle \frac{\partial^2  u}{\partial x_{i}\partial x_{j}}, f \rangle_{\mathscr{D}'(\Omega) \times \mathscr{D}(\Omega)} \vert \leq C_{ij}(\Omega) \Vert \frac{\partial u}{\partial x_{i}} \Vert_{L^{2}(\Gamma)} \Vert \chi_j \Vert_{H^{1}(\Gamma)} \leq C_{ij}(\Omega) \Vert \frac{\partial u}{\partial x_{i}} \Vert_{L^{2}(\Gamma)} \Vert \chi_j \Vert_{H^{3/2}(\Omega)}.  
$$
From the estimate in \eqref{estimNeuj}, we finally deduce that
$$
\vert \langle \frac{\partial^2  u}{\partial x_{i}\partial x_{j}}, f \rangle_{\mathscr{D}'(\Omega) \times \mathscr{D}(\Omega)} \vert\leq C_{ij}(\Omega) \Vert u \Vert_{H^{3/2}(\Omega)} \Vert f  \Vert_{H^{1/2}(\Omega)}.
$$
and then the estimate
$$
\Vert \frac{\partial^2  u}{\partial x_{i}\partial x_{j}}\Vert_{[H^{\, 1/2}(\Omega)]'} \leq C_{ij}(\Omega) \Vert u \Vert_{H^{3/2}(\Omega)} 
$$
by using the density of $\mathscr{D}(\Omega)$ in $H^{\, 1/2}(\Omega)$. This proves the second part of implication \eqref{impharC1132}. \medskip

The isomorphism $\gamma_{\textit{\textbf n}}  :  H^{3/2}(\Omega)\cap L^2_0(\Omega) \cap \mathscr{H} \rightarrow L^2_0(\Gamma)$ is a simple consequence of the fact that for every $h\in L^2_0(\Gamma)$, there exists a unique harmonic function $w\in H^{3/2}(\Omega)\cap L^2_0(\Omega) $ in $\Omega$ satisfying $ \partial_\textit{\textbf n}  w = h$ on $\Gamma$  (see \cite{Jer1}). And we have the following equivalence:
$$
 \Vert v \Vert_{H^{3/2}(\Omega)} \approx \Vert \gamma_{\textit{\textbf n}} v \ \Vert_{L^{2}(\Gamma)} \quad \mathrm{for}\quad  v\in H^{3/2}(\Omega)\cap L^2_0(\Omega) \cap \mathscr{H}.
 $$
 And we clearly have the identity  $\gamma_\textit{\textbf n}(E(\nabla^2, \Omega)) = L^2_0(\Gamma)$.
\end{proof}

\bigskip

$^\star$ This paper is dedicated to the memory of my dear friend Mohand Moussaoui, who passed away before the work was completed.
%

\end{document}